\documentclass[11pt]{article}

\usepackage{amssymb}
\usepackage{amsthm}
\usepackage{amsmath}
\usepackage{color}
\usepackage[dvipsnames]{xcolor}
\usepackage{graphicx}

\usepackage{subfig}
\usepackage{txfonts}
\usepackage{cases}
\usepackage{mathrsfs}
\usepackage{ebezier}
\usepackage{epic,eepic}
\usepackage{epsfig,mathrsfs}

\newtheorem{theorem}{Theorem}[section]

\newtheorem{lemma}[theorem]{Lemma}

\newcommand{\tr}{\operatorname{tr}}

\newcommand{\hid}[1]{}

\newcommand{\ourtitle}{A signal separation technique for sub-cellular imaging using dynamic optical coherence tomography}
\title{\ourtitle}

\author{Habib Ammari\thanks{\footnotesize Department of Mathematics,
ETH Z\"urich,
R\"amistrasse 101, CH-8092 Z\"urich, Switzerland (habib.ammari@math.ethz.ch, francisco.romero@sam.math.ethz.ch).} \and Francisco Romero\footnotemark[1]  \and Cong Shi\thanks{\footnotesize Computational Science Center,
University of Vienna, Oskar Morgenstern-Platz 1, 1090 Vienna, Austria (cong.shi@univie.ac.at).}}

\usepackage{dsfont}
\begin{document}
\date{}
\maketitle

\begin{abstract}
This paper aims at imaging the dynamics of metabolic activity of cells. Using dynamic optical coherence tomography, we introduce a new multi-particle dynamical model to simulate the movements of the collagen and the cell metabolic activity and develop an efficient signal separation technique for sub-cellular imaging. We perform a singular-value decomposition of the dynamic optical images to isolate the intensity of the metabolic activity. We prove that the largest eigenvalue of the associated Casorati matrix corresponds to the collagen. We present several numerical simulations to illustrate and validate our approach. \end{abstract}

\medskip

\bigskip

\noindent {\footnotesize Mathematics Subject Classification
(MSC2000): 92C55, 78A46, 65Z05}


\section{Introduction}

Since dynamic properties are essential for a disease prognosis and a selection of treatment options, a number of methods to explore these dynamics has been developed. When optical imaging methods are used to observe cell-scale details of a tissue, the highly-scattering collagen usually dominates the signal, obscuring the intra-cellular details. A challenging problem is to remove the influence of the collagen in order to have a better imaging inside the cells.

There have been many studies on optical imaging to extract useful information. In \cite{lee2012dynamic} the authors use stochastic method, which follows from a probabilistic model for particle movements, and then they express the autocorrelation function of the signal in terms of some parameters including different components of the velocity and the fraction of moving particles. Those parameters are then estimated using a fitting algorithm. In \cite{joo2010diffusive, li2009intracellular}, the autocorrelation function of the signal can be written as a complex-valued exponential function of the particle displacements. Through the relation between the real and imaginary parts of this autocorrelation function, the authors analyze the temporal autocorrelation on the complex-valued signals to obtain the mean-squared displacement (MSD) and also time-averaged displacement (TAD) (which is the velocity) of scattering structures. Very recently, in \cite{apelian2016dynamic}, Apelian et. al. use difference imaging method, which consists in directly removing the stationary parts from the images by taking differences or standard deviations. The motivation of this paper comes from \cite{apelian2016dynamic}.

Some researchers use Doppler optical coherence tomography to obtain high resolution tomographic images of static and moving constituents simultaneously in highly scattering biological tissues, for example, \cite{chen1999optical} and in \cite[Chapter 21]{drexler2008optical}.

In this paper, using dynamic optical coherence tomography we introduce a signal separation technique for sub-cellular imaging and give a detailed mathematical analysis of extracting useful information. This includes giving a new multi-particle dynamical model to simulate the movement of the collagen and metabolic activity, and also providing some results relating the eigenvalues and the feasibility of using singular value decomposition (SVD) in optical imaging, which as far as we know is original. %This is also the first paper to give theoretical justification of using SVD.

The paper has three main contributions. First, we give a new model as an extension of the single particle optical Doppler tomography, which allows us to justify the SVD approach for the separation between the collagen signal and metabolic activity signal. Then we perform eigenvalue analysis for the operator with the intensity as an integral kernel, and prove that the largest eigenvalue corresponds to the collagen. This means that using a SVD of the images and removing the part corresponding to the largest eigenvalue is a viable method for removing the influence of collagen signals. Finally, based on SVD, we give a new method for isolating the intensity of the metabolic activity.

The paper is structured as follows. In Section 2 we introduce our multi-particle dynamical model based on a classical model in \cite{drexler2008optical}. In Section 3, we discuss the forward operator with total signal as its integral kernel, and give its eigenvalue analysis, showing that the part corresponding to the collagen signal have rank one, which provides the theoretical foundation for using SVD. In Section 4, we discuss the mathematical rationality for using a SVD method and the method of isolating the metabolic signal. In Section 5 we give some numerical experiments. Some concluding remarks are presented in the final section.

\section{The dynamic forward problem}

Optical Coherence Tomography (OCT) is a medical imaging technique that uses light to capture high resolution images of biological tissues by measuring the time delay and the intensity of backscattered or back reflected light coming from the sample. The research on OCT has been growing very fast for the last two decades. We refer the reader, for instance, to \cite{huang1991optical, fercher1996optical, fercher2003optical, podoleanu2014optical, schmitt1999optical, tomlins2005theory}.
This imaging method has been continuously improved in terms of speed, resolution and sensitivity. It has also seen a variety of extensions aiming to assess functional aspects of the tissue in addition to morphology. One of these approaches is Doppler OCT (called ODT), which aims at visualizing movements in the tissues (for example, blood flows). ODT lies on the identical optical design as OCT, but additional signal processing is used to extract information encoded in the carrier frequency of the interferogram.
%Such extensions were already implemented in time-domain systems, but have gained importance with the introduction of Fourier domain OCT. Nowadays, phase-sensitive detection techniques are most widely used to extract blood velocity and blood flow from tissues, for example \cite{leitgeb2014doppler}.

The purpose of this paper is to analyze the mathematics of ODT in the context of its application for imaging sub-cellular dynamics. We prove that a signal separation technique performs well and allows imaging of sub-cellular dynamics. We refer the reader to \cite{separation1,separation2,separation3} for recently developed signal separation approaches in different biomedical imaging frameworks. These include ultrasound imaging, photoacoustic imaging, and electrical impedance tomography.

\subsection{Single particle model}

We first consider a single moving particle. In  \cite[Chapter 21]{drexler2008optical}, the optical Doppler tomography is modeled as follows. Assume that there is one moving particle at a point $x$ in the sample $\Omega$ and denote by $\nu$ the z-component of its velocity. Then the ODT signal generated by this particle is given by
\begin{equation}\label{originalmodel}
\Gamma_{ODT}(x, t)=2\int_{0}^{\infty}S_{0}(\omega)K(x, \omega)K_{R}(x, \omega)\cos(2\pi\omega (\tau+\frac{\Delta}{c})+2\pi\omega \frac{2\bar{n}vt}{c})d\omega,
\end{equation}
where $\omega$ is the frequency, $S_0(\omega)$ is the spectral density of the light source, $K_(x, \omega)$ and $K_{R}(x, \omega)$ are the reflectivities of the sample and the reference mirror respectively, $\bar{n}$ is the index of refraction, $c$ is the speed of the light, $\tau$ is the time delay on the reference arm, and $\Delta$ is the path difference between the reference arm and sample arm.

Since $\cos$ is an even function, the above integral can be rewritten as
\begin{equation}\label{single_particle}
\Gamma_{ODT}(x, t)=\int_{-\infty}^{\infty}S_{0}(\omega)K(x, \omega)K_{R}(x, \omega)e^{2\pi\omega i(\tau+\frac{\Delta}{c})+2\pi\omega i\frac{2\bar{n}vt}{c}}d\omega .
\end{equation}

\begin{figure}[h!]
	\centering
	\begin{minipage}[c]{.45\textwidth}
		\centering
		\tiny a
	\end{minipage}
	\hspace{.5cm}
	\begin{minipage}[c]{.45\textwidth}
		\centering
		\tiny b
	\end{minipage}
	\\
	\begin{minipage}[c]{.55\textwidth}
		\centering
		\includegraphics[width=\linewidth]{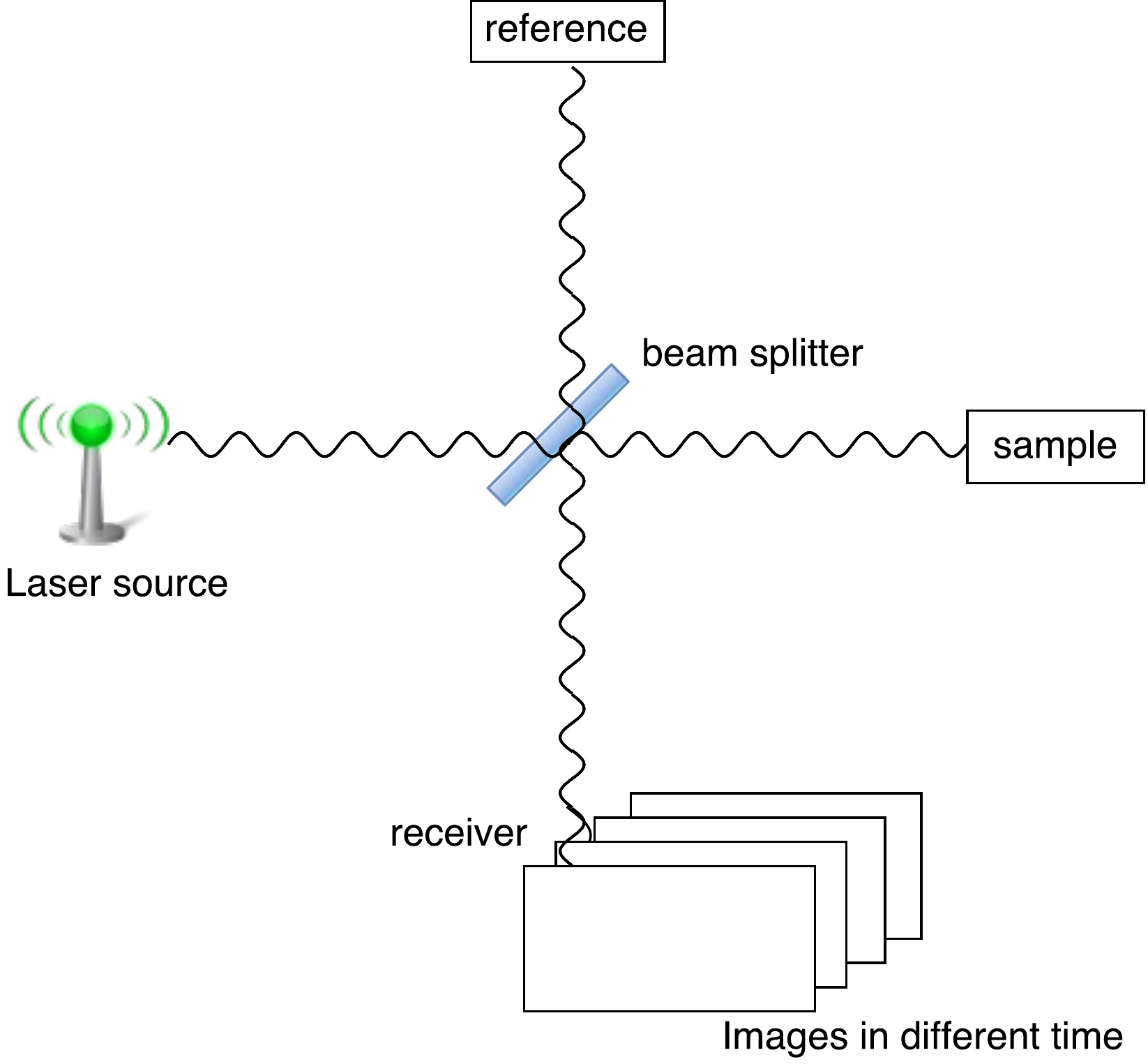}
	\end{minipage}
	\hspace{.5cm}
	\begin{minipage}[c]{.32\textwidth}
		\centering
		\includegraphics[width=\linewidth]{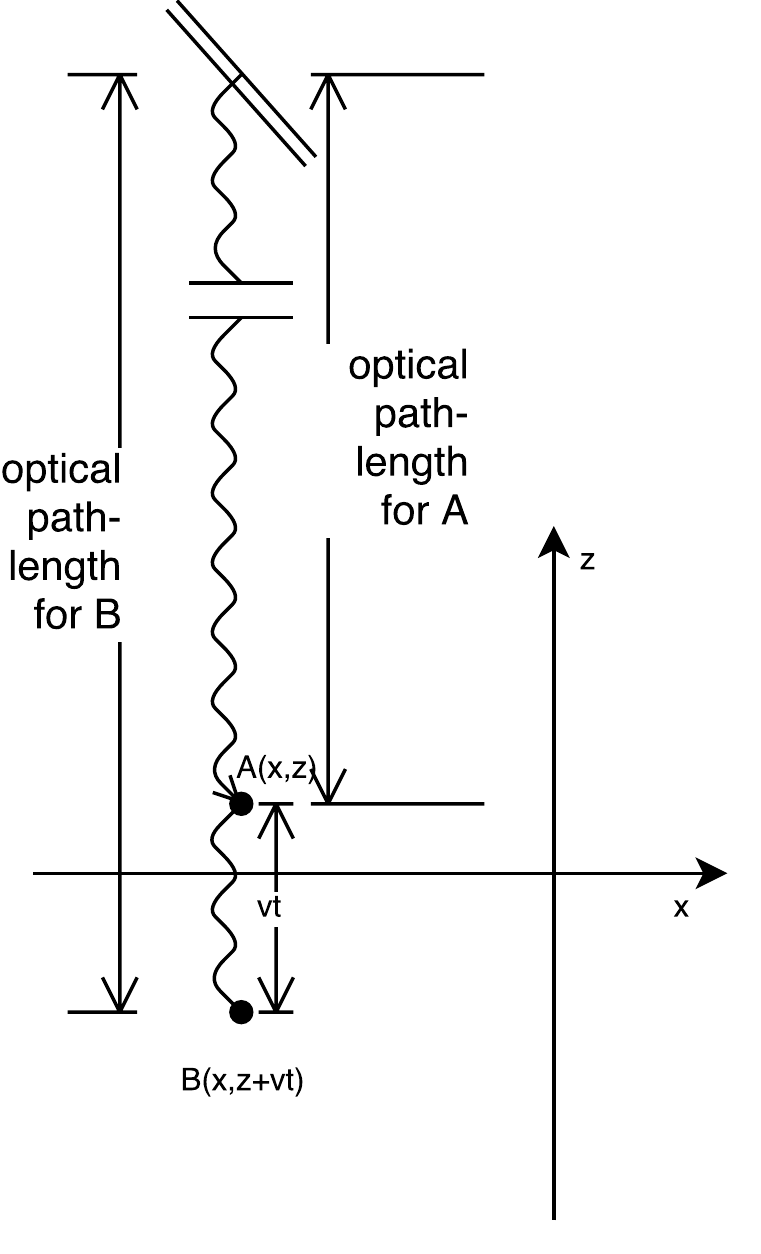}
	\end{minipage}
	\caption{a) Illustration of the imaging setup.
b)  A particle moves from $A$ to $B$ covering a distance of $vt$. When the particle is at $B$, the light travels an additional distance of $2vt$ inside a medium with refrative index $\bar{n}$, so the effective path-length of the sample arm increases by $2\bar{n}vt$.}
	\label{fig:singleparticle}
\end{figure}

To give an explanation for the exponential term of the above formula, we choose a suitable coordinate system such that the beam propagates along the $z$-direction, and suppose that the particle moves in this direction from point $A$ to point $B$ with velocity $v$, which also means covering a distance of $vt$ (see Figure \ref{fig:singleparticle}). Physically, the received signal $\Gamma_{ODT}$ is determined by the effective path-length difference between the sample and reference arms. In addition, for this moving particle the effective path-length difference is represented by the quantity $c\tau+\Delta+2\bar{n}vt$, which could also be seen as the z-coordinate of the particle (see Figure \ref{fig:singleparticle}).

Note that (\ref{single_particle}) is only applicable to a single particle at $x$ moving with a constant velocity $v$. For a particle with a more general movement, the path-length difference is no longer a linear function with respect to $t$. Nevertheless, we define $\varphi(t)$ as the $z$-coordinate of the particle at time $t$, which is a generation of $c\tau+\Delta+2\bar{n}vt$. Also, in our case the reference arm is a mirror, so without loss of generality, we make the assumption that $K_{R}(x, \omega)=1$. Then the following expression for signal $\Gamma_{ODT}(x, t)$ holds
\begin{equation*}\label{single_particle_2}
\Gamma_{ODT}(x, t)=\int_{-\infty}^{\infty}S_{0}(\omega)K(x, \omega)e^{2\pi\omega i(\frac{2\bar{n}}{c}\varphi(t))}d\omega.
\end{equation*}

This is not just a simplification of the model (\ref{originalmodel}), but also a small modification, since the particles with regular and random movements produce difference signals. Here we look into more details of particle movements. For the sake of simplicity, we assume that the collagen particles move with a constant speed $v$, so $\varphi(t)=\varphi(0)+vt$. On the other hand, for the particles belonging to the metabolic activity part, $\varphi(t)$ behaves as a random function, since we do not have much information with regard to them.

{\bf Remark.} Formula (\ref{single_particle_2}) is derived in \cite{drexler2008optical} by considering what is essentially our $\phi(t)$ (written as $\Delta_d$ there, see formula (21.11) and (21.15) of \cite{drexler2008optical}.) This justifies our treatment for general particles above. We emphasize that we generalized the model in \cite{drexler2008optical} to accommodate particles with variable velocities.

\subsection{Multi-particle dynamical model}

We have seen the effect of the image $\Gamma_{ODT}(x, t)$ for one moving particle. We now consider the more realistic case of a medium (could be cell or tissue) with a large number of particles in motion. In actual imaging, for each pixel which we denote also by $x$, there would be many particles, all with different movement patterns.

We choose an appropriate coordinate system, such that for any particle on the plane $z=0$, its effective path-length difference is zero. Let $L$ be the coherence length. Physically, only the particles with path-length difference smaller than $L$, or equivalently $z\in[-L,L]$, will be present in the image. In fact, if the differences between the two arms are larger than the coherence length, then the lights from two arms do not interfere anymore, and thus do not contribute to the received signal. This means that the imaging region is a "thin slice" within the sample with thickness $2L$ (see Figure \ref{fig:multiparticle}). Then we divide the slice into small regions, such that each region corresponds to a pixel of the final image. See Figure \ref{fig:multiparticle} for the imaged small region, which is given by $x\times[-L,L]$, and for the correspondence between them and pixels of the final image.

\begin{figure}[h!]
	\centering
	\includegraphics[width=0.8\linewidth]{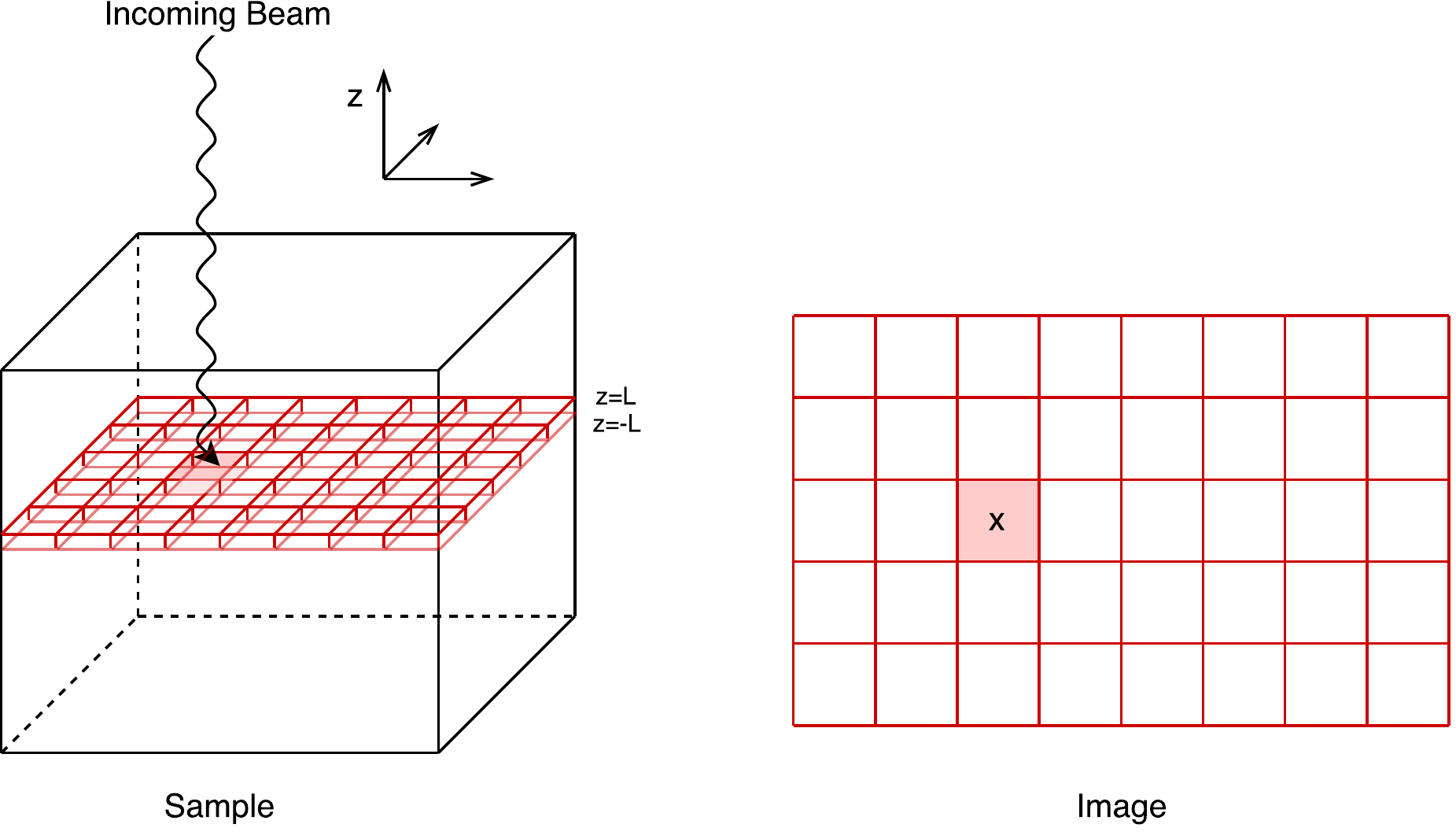}
	\caption{One "slice" in the sample, and its division into small regions corresponding to the pixels of the image.}
	\label{fig:multiparticle}
\end{figure}

Since there are many particles in this region, we describe their distribution using a density function $p$. Moreover, for any function $f(z)$, we have that the integral $\int^{z_2}_{z_1}f(z)p(x, z, t) dz$ is equal to the sum of $f(z)$ over all particles in $x\times [z_1,z_2]$. We know that the received light intensity in the small region $x\times[-L,L]$ could be seen as the sum of light intensity over all particles in this region. Therefore for uniform medium, we can write it as an integral in terms of the density function $p(x,z,t)$,

\begin{equation*}
\Gamma_{ODT}(x, t)=\int_{-\infty}^{\infty}\int_{-L}^{L}S_{0}(\omega)K(x, \omega)e^{2\pi\omega i(\frac{2\bar{n}}{c}z)}p(x, z, t)d\omega \,dz,
\end{equation*}
noting that the reflectivity coefficient $K$ must be the same for all involved particles.
According to the definition of $p(x, z, t)$, we consider it as the sum of the density function of collagen particles and the density function of metabolic activity particles, namely,
\begin{equation}\label{psplit}
p(x, z, t)=p_c(x, z, t)+p_m(x, z, t).
\end{equation}
Consequently, their respective reflectivities will be denoted $K_c$ and $K_m$, giving us the ODT measurements formula
\begin{equation}\label{gammasplit}
\Gamma_{ODT}(x, t)=\Gamma^c_{ODT}(x, t)+\Gamma^m_{ODT}(x, t),
\end{equation}
where $\Gamma^c_{ODT}(x, t)$ corresponds to the collagen signal and $\Gamma^m_{ODT}(x, t)$ corresponds to the metabolic activity signal, with formulas
\begin{equation}\label{smallareamodel}
\Gamma^j_{ODT}(x, t)= \int_{-\infty}^{\infty}\int_{-L}^{L}S_{0}(\omega) K_j(x, \omega)p_j(x, z, t)e^{2\pi\omega i(\frac{2\bar{n}}{c}z)}d\omega \,dz, \quad \text{for }j\in \left\{ c,m \right\}.
\end{equation}
Physically, since the collagen moves as a whole, we could assume that the collagen particles move with one uniform (and very small) velocity $v_0$, which means any such particles will be at position $z+v_0t$ at time $t$. Let $q_c(x, z)$ denote the density function of all the collagen particles inside area $x$ with initial vertical position $z$. Then we have
\begin{equation}\label{collagen-move}
p_c(x, z+v_0t, t)=q_c(x, z).
\end{equation}
Furthermore, from this expression we could see when $t=0$, $q_c(x,z)=p_c(x,z,0)$.

In the case of metabolic activity we do not assume any conditions on the density function $p_m(x, v, z)$, because there is no physical law of motions for us to use. In the numerical experiments, because of the large number of particles, a random medium generator is used to simulate the particle distribution while keeping the computational cost low.

Since $x$ is a small area inside the sample, when we choose $x$, it could include both collagen particles and metabolic activity particles. The aim is to separate the two classes of particles. In practice, the contributions of collagen particles to the intensity is much larger than the contributions of the metabolic activity. This allows us to understand that the reflectivity of collagen particles $K_c$ is much larger (realistic quantities are about $10^2$ to $10^4$ times) than the reflectivity of metabolic activity particles $K_m$, and
\begin{equation}\label{gamma_signal_compare}
|\Gamma^c_{ODT}(x, t)|\gg |\Gamma^m_{ODT}(x, t)|.
\end{equation}

In this section, we have given a multi-particle dynamical model, to separate the collagen signal and the metabolic activity signal. The next step is to analyze the properties of this model.

\section{Property analysis of the forward problem}

\subsection{Direct operator representation}

Based on the multi-particle dynamical model, in order to analyze the properties of collagen and metabolic activity, we first represent their corresponding operators.

Let $S$ be the integral operator with the kernel $\Gamma_{ODT}(x, t)$, which is a real-valued function given by (\ref{smallareamodel}). The collagen signal has high correlation between different points, while the metabolic signals have relatively lower correlation, so it would be useful to look at the correlation of the whole signal.
The correlation between two points $x$ and $y$ can be represented as $\int \Gamma_{ODT}(x, t)\overline{\Gamma_{ODT}(y, t)}\,dt$, which is exactly the integral kernel of the operator $SS^*$, where $S^*$ is the adjoint operator of $S$. We denote the kernel of $SS^*$ by
%Also, the eigenvalues of the operator $SS^*$ are squares of the singular values of $S$, and the eigenvectors of $SS^*$ coincide with space singular vectors of $S$.
\begin{equation}\label{kernelForig}
F(x, y)= \int_{0}^{T}\Gamma_{ODT}(x, t)\overline{\Gamma_{ODT}(y, t)}dt,
\end{equation}
for some fixed $T>0$. Substituting the representation of $\Gamma_{ODT}(x, t)$ in (\ref{gammasplit}) into (\ref{kernelForig}), we arrive to
\begin{align*}
F(x, y)=F_{cc}(x, y)+F_{cm}(x, y)+F_{mc}(x, y)+F_{mm}(x, y),
\end{align*}
where for $j, k \in \{ c, m\} $, $F_{jk}(x, y)$ is given by
\begin{align}\label{Fij}
\begin{split}
F_{jk}(x, y)&=\int_{\mathbb{R}^{2}\times[-L,L]^2\times[0, T]}S_{0}(\omega_1)S_{0}(\omega_2)K_j(x, \omega_1)K_k(y, \omega_2)p_j(x, z_1, t)\\
&\times p_k(y, z_2, t)e^{\frac{4\pi i\bar{n}}{c}(\omega_1 z_1-\omega_2 z_2)}d\omega_1 d\omega_2 dz_1dz_2dt,
\end{split}
\end{align}
with $z_1, z_2\in [-L,L]$ and $\omega_1, \omega_2 \in \mathbb{R}, t \in [0,T]$.
Likewise, we denote the corresponding operator by $S_jS^*_{k}$ for $j,k \in \{ c,m \}$.
In the case of the collagen signal, note that the operator $S_cS_c^*$ contains the solely collagen information.

 First we consider its kernel $F_{cc}$. Applying the uniform movements of collagen particles (\ref{collagen-move}) along the $z$-direction yields
\begin{align}\label{kernelFdmm0000}
\begin{split}
F_{cc}(x, y)&=\int_{\mathbb{R}^{2}\times[-L,L]^2\times[0, T]}S_{0}(\omega_1)S_{0}(\omega_2)K_c(x, \omega_1)K_c(y, \omega_2)q_c(x,z_1-v_0t)\\
            &\times q_c(y,z_2-v_0t)e^{\frac{4\pi i\bar{n}}{c}(\omega_1 z_1-\omega_2 z_2)}d\omega_1 d\omega_2 dz_1dz_2dt.
            \end{split}
\end{align}
%&=2\int_{\mathbb{R}^{3}\times[-L,L]^2}S_{0}(\omega_1)S_{0}(\omega_2)K_{S}(x, \omega_1)K_{S}(y, \omega_2)p_c(x,z_1,t)p_c(y,z_2,t)\\
%            &\times e^{\frac{4\pi i\bar{n}}{c}(\omega_1 z_1-\omega_2 z_2)}d\omega_1 d\omega_2 dz_1dz_2dt\\

In order to simplify this expression even further, let us introduce a couple of assumptions.

Physically, since the scale of collagen and inter-cellular structures (such as collagen) are much larger than the coherence length $L$, the particle distribution inside a small slice $|z|<L$ should be more or less uniform. Therefore, it is reasonable to assume that $q_c(x, z)$ does not actually depend on $z$ inside such a small slice, namely, $q_c(x, z)=q_c(x)$.

Furthermore, in practice the tissue being imaged is nearly homogeneous, and therefore the reflectivity spectrum, (or more intuitively, the "color" of the tissue) should stay the same everywhere. The only difference in reflectivity between two points should be a difference of total reflectivity (using our "color" analogy, the two points would look like, e.g. "different shades of red", and not "red and yellow"). Therefore, for any two pixels $x_1$ and $x_2$, by looking at the reflectivities $K_c(x_1, \omega)$ and $K_c(x_2, \omega)$ as functions of frequency $\omega$, they are directly proportional. Thus it is reasonable to assume that $K_c(x, \omega)$ could be written in the variable separation form $K_{c_1}(x)K_{c_2}(\omega)$.

Under these two assumptions, the expression of $F_{cc}(x,y)$ can be simplified considerably:
\begin{equation}\label{kernelFdmm}
\begin{split}
F_{cc}(x, y)&=TK_{c_1}(x)K_{c_1}(y)q_c(x)q_c(y)\\
            &\times \int_{\mathbb{R}^{2}\times[-L,L]^2}S_{0}(\omega_1)S_{0}(\omega_2)K_{c_2}(\omega_1)K_{c_2}(\omega_2)e^{\frac{4\pi i\bar{n}}{c}(\omega_1 z_1-\omega_2 z_2)}d\omega_1 d\omega_2 dz_1dz_2\\
            &=TK_{c_1}(x)K_{c_1}(y)q_c(x)q_c(y)\\
            &\times\int_{[-L,L]^2}\mathcal{F}(S_{0}K_{c_2})(-\frac{4\pi \bar{n}z_1}{c})\mathcal{F}(S_{0}K_{c_2})(\frac{4\pi \bar{n}z_2}{c})dz_1dz_2\\
\end{split}
\end{equation}
where the Fourier transform of a function $f(\omega)$ is defined as $\mathcal{F}f(\tau)=\int_{\mathbb{R}}f(\omega)e^{-i\omega\tau}d\omega$.

This is the fundamental formula for analyzing collagen signal, since from this formula, we could see that $F_{cc}(x, y)$ is variable separable with respect to $x$ and $y$. This property gives us a hint to compute the eigenvalues of the collagen signal.

For the correlation terms $F_{cm}(x, y)$ and $F_{mc}(x, y)$, which contains both the collagen  and metabolic activity signals, we use again the uniform movement assumption for $p_c$ while keeping the metabolic part $p_m$. Inserting (\ref{collagen-move}) into (\ref{Fij}), we have

\begin{align}\label{kernelFdmc}
\begin{split}
F_{mc}(x, y)&=K_{c_1}(y)q_c(y)\int_{[-L,L]^2\times[0, T]}\mathcal{F}(S_{0}K_{c_2})(\frac{4\pi \bar{n}z_2}{c})\\
&\qquad \times \mathcal{F}(S_{0}K_{m})(x,-\frac{4\pi \bar{n}z_1}{c})p_m(x, z_1, t)dz_1dz_2 dt,
\end{split}
\end{align}
and
\begin{align}\label{kernelFdcm}
\begin{split}
F_{cm}(x, y)&=K_{c_1}(x)q_c(x)\int_{[-L,L]^2\times[0, T]}\mathcal{F}(S_{0}K_{c_2})(-\frac{4\pi \bar{n}z_1}{c})\\
&\qquad \times \mathcal{F}(S_{0}K_{m})(y,\frac{4\pi \bar{n}z_2}{c})p_m(y, z_2, t)dz_1dz_2 dt.
\end{split}
\end{align}
From representations (\ref{kernelFdmc}) and (\ref{kernelFdcm}), we could see that $F_{mc}(x, y)$ and $F_{cm}(x, y)$ have also variable separated forms with respect to $x$ and $y$.

In the case of the metabolic activity kernel $F_{mm}(x,y)$, by keeping the representation $p_m$, it is clear that
\begin{align}\label{kernelFdcc}
\begin{split}
F_{mm}(x,y)&=\int_{[-L,L]^2\times[0, T]}\mathcal{F}(S_{0}K_{m})(x,-\frac{4\pi \bar{n}z_1}{c})\mathcal{F}(S_{0}K_{m})(y,\frac{4\pi \bar{n}z_2}{c})\\
&\times p_m(x, z_1 ,t)p_m(y, z_2, t)dz_1dz_2 dt.
\end{split}
\end{align}

To sum up, the main feature of our multi-particle dynamical model is that, except the sole metabolic activity signal, all the other parts have kernels of variable separable form. Therefore, it is important to relate this property to the separation of the signals. This will be the aim of the next subsection.

\subsection{Eigenvalue analysis}

We have given the representation of the integral operators and their corresponding kernels. In order to argue for the feasibility of using a SVD, we will calculate the corresponding eigenvalues, showing that the collagen signal has one very large eigenvalue relative to the metabolic activity. We assume that the eigenvalues are ordered decreasingly, so $\lambda_1$ is the largest one.

We first recall that for an operator $A$ with rank one, the unique non-zero eigenvalue $\lambda$ is equal to the trace of $A$. From the expression of $F_{cc}(x, y)$, we could see that $F_{cc}(x, y)$ has rank one because of the separable form with respect to $x$ and $y$, so the operator $S_cS^*_{c}$ has only one nonzero eigenvalue, which we denote by $\lambda(S_cS^*_{c})$. Now we compare $\lambda(S_cS^*_{c})$ and the eigenvalues of the operator $S_mS^*_{m}$.

%\begin{lemma}\label{eigenvalue_mm}
%The operator $S_cS^*_{c}$ has only one positive eigenvalue that satisfies
%\begin{equation}
%\lambda(S_cS^*_{c})=\frac{c}{2\pi \bar{n}v_0}\|S_0K_{S_2}\frac{\sin(2\pi L \omega)}{\pi\omega}\|_{L^2(\mathbb{R})}^2 \|K_{S_1}(x)q_c(x)\|_{L^2(\Omega)}^2.
%\end{equation}
%\end{lemma}
%\begin{proof}
%Let $\tr$ denote the trace. Since $F_{cc}(x, y)$ defined in (\ref{kernelFdmm}) is the kernel of operator $S_cS^*_{c}$ and has rank one, we have that $\lambda(S_cS^*_{c})=\tr (S_cS^*_{c})$. Hence,
%\begin{align*}\label{traceFdmm}
%\begin{split}
%\tr (S_cS^*_{c})&=\int_{x\in \Omega}F_{cc}(x, x)dx\\
%&=\frac{c}{2\pi \bar{n}v_0} \|S_0K_{S_2}\frac{\sin(2\pi L \omega)}{\pi\omega}\|_{L^2(\mathbb{R})}^2 \int_{x\in \Omega}K_{S_1}(x)K_{S_1}(x)q_c(x)q_c(x)dx\\
%&=\frac{c}{2\pi \bar{n}v_0}\|S_0K_{S_2}\frac{\sin(2\pi L \omega)}{\pi\omega}\|_{L^2(\mathbb{R})}^2 \|K_{S_1}(x)q_c(x)\|_{L^2(\Omega)}^2.
%\end{split}
%\end{align*}
%
%\end{proof}

\begin{lemma}\label{compare_mm_cc}
Let $S_cS^*_{c}$ and $S_mS^*_{m}$ be the integral operators with kernels $F_{cc}$ and $F_{mm}$ defined in (\ref{kernelFdmm}) and (\ref{kernelFdcc}), respectively. If the intensities of collagen and metabolic activity satisfy (\ref{gamma_signal_compare}), then we have
\begin{equation*}
\lambda(S_cS^*_{c}) \gg \lambda_i(S_mS^*_{m}),\,\,\, \forall i\geq 1.
\end{equation*}
\end{lemma}
\begin{proof}

On one hand, $F_{cc}(x, y)$ has rank one, so it is clear that
\begin{equation}\label{trmm}
\lambda(S_cS^*_{c})=\tr(S_cS^*_{c}).
\end{equation}

On the other hand, since the eigenvalues of operator $S_mS^*_{m}$ are all positive, any eigenvalue $\lambda_i(S_mS^*_{m})$ satisfies
\begin{equation}\label{trcc}
\lambda_i(S_mS^*_{m})<\Sigma_{i=1}^{\infty}\lambda_i(S_mS^*_{m})=\tr(S_mS^*_{m}).
\end{equation}

Then it suffices to prove that $\tr(S_cS^*_{c})\gg \tr(S_mS^*_{m})$.
%To prove this, we use the condition \ref{gammasplit}.
From the definition of trace of an operator, we readily get $\tr(S_cS^*_{c})=\int_{x\in \Omega}F_{cc}(x, x)dx$. Substituting the expression (\ref{kernelForig}) into the above formula yields
\begin{align*}
\begin{split}
\tr(S_cS^*_{c})&=\int_{x\in \Omega}\frac{1}{2}\int_{-\infty}^{\infty}\Gamma^c_{ODT}(x, t)\overline{\Gamma^c_{ODT}(x, t)}dtdx\\
&=\frac{1}{2}\int_{x\in \Omega}\int_{-\infty}^{\infty}|\Gamma^c_{ODT}(x, t)|^2dtdx.
\end{split}
\end{align*}

The same analysis can be carried out by looking at $\tr(S_mS^*_{m})$,
\begin{equation*}
\tr(S_mS^*_{m})=\frac{1}{2}\int_{x\in \Omega}\int_{-\infty}^{\infty}|\Gamma^m_{ODT}(x, t)|^2dtdx.
\end{equation*}

Recall that the intensity of collagen signal is much larger than metabolic activity signal, which is the assumption in (\ref{gamma_signal_compare}). Hence, we obtain the trace comparison $\tr(S_cS^*_{c})\gg \tr(S_mS^*_{m})$.
%\begin{equation*}
%\tr(SS^*_{cc})\gg \tr(SS^*_{mm}).
%\end{equation*}

%Combining (\ref{trmm}) and (\ref{trcc}), we have $\lambda(SS^*_{cc}) \gg \lambda(SS^*_{mm})$.
\end{proof}

Now we compare the eigenvalue $\lambda(S_cS^*_{m})$ with $\lambda(S_cS^*_{c})$ and $\lambda_1(S_mS^*_{m})$.

\begin{lemma}\label{compare_mc_cc_mm}
Let $S_cS^*_{c}$, $S_cS^*_{m}$ and $S_mS^*_{m}$ be the integral operators with kernels defined in (\ref{kernelFdmm}), (\ref{kernelFdmc}) and (\ref{kernelFdcc}). Then their eigenvalues satisfy
\begin{equation*}
\lambda_i(S_cS^*_{m})\leq \sqrt{\lambda(S_cS^*_{c})\lambda_1(S_mS^*_{m})}
\end{equation*}
for all $i$.
\end{lemma}
\begin{proof}
Recall the definition of the operator norm of an operator $A$, namely, $\|A\|_{OP}=\sup\{\frac{\|Av\|}{\|v\|}, v\in V \,\text{with}\, v\neq 0\}$, which yields $\lambda(S_cS^*_{m})\leqslant \|S_cS^*_{m}\|_{OP}$. Since the operator norm is equal to the largest singular value, direct calculation shows that
\begin{equation*}
\begin{split}
\|S_cS_m^*\|_{OP}&\leq\|S_c\|_{OP}\|S_m^*\|_{OP}\\
&=\sigma_1(S_c)\sigma_1(S_m^*)\\
&=\sqrt{\lambda(S_cS_c^*)}\sqrt{\lambda_1(S_mS_m^*)} ,
\end{split}
\end{equation*}
where $\sigma_1$ denotes the largest singular value.
\end{proof}

%Let $S_c$ and $S_m$ be the integral operators with kernel $\Gamma_{ODT}^c(x,t)$ and $\Gamma_{ODT}^c(x,t)$ respectively, so it is clear that $SS^*_{ij}=S_iS_j^*$ for any $i,j\in\{c,m\}$.

In this section, we discussed eigenvalue analysis for the forward operator of multi-particle dynamical model. More explicitly, we showed that the largest eigenvalue corresponds to the collagen signal, the middle eigenvalues mix the collagen signal and metabolic activity signal, and the remaining eigenvalue corresponds to the metabolic activity signal. Also in our model the solely collagen signal has rank one, which provides a good reason to use SVD in solving the inverse problem.

%Until now, we only get $\lambda(SS^*_{cm})$ is between $\lambda(SS^*_{cc})$ and $\lambda(SS^*_{mm})$, but we don't know if it will be close to the upper bound or near zero.

\section{The inverse problem: Signal separation}

Our main purpose in this paper is to image the dynamics of metabolic activity of cells. Highly backscattering structures like collagen dominate the dynamic OCT signal, masking low-backscattering structures such as metabolic activity. As shown in the modeling part, we divide the scattering particles in the tissue into the high-backscattering collagen part, and the low-backscattering metabolic activity part. Based on this division, the resulting image $\Gamma_{ODT}(x, t)$ could also be written as the sum of the collagen $\Gamma^c_{ODT}(x, t)$ and the metabolic activity part $\Gamma^m_{ODT}(x, t)$. The inverse problem is to recover the intensity of metabolic activity of cells from the image $\Gamma_{ODT}(x, t)$. In this paper, we use singular value decomposition (SVD) method to approximate the metabolic activity part, then using a particular formula (see (\ref{meq-intensity})) to get its corresponding intensity.

\subsection{Analysis of SVD algorithm}

Since we have proved the high backscattering signal corresponds to a rank one kernel and this part is far larger than the rest, It is natural to associate it to the first singular value of the SVD expansion. We claim that in order to remove the high backscattering signal, it is reasonable to remove the first term in the SVD expansion of the image. In this section we first recall the SVD algorithm, and then we assert that there is a gap between our model and SVD algorithm. At the end of this section, we give a result to illustrate the gap is small that we could ignore.

Let $x_1,x_2,\dots, x_j, \dots$ denote the pixels of the image. We define the matrices $A,A_c\in \mathbb{C}^{n_x\times n_t}$ by
\begin{equation*}\label{AAc}
\begin{split}
A_{j, k}&=\Gamma_{ODT}(x_j, t_k)\\
(A_c)_{j, k}&=\Gamma_{ODT}^c(x_j, t_k),
\end{split}
\end{equation*}
where $j\in\{1,...,n_x\}, \, k\in\{1,...,n_t\}$.

Recall that a non-negative real number $\sigma$ is a singular value for a matrix $A$, if and only if there exists unit vectors $u\in \mathbb{R}^{n_x}$ and $v\in\mathbb{R}^{n_t}$ such that
\begin{equation*}
Av=\sigma u   \,\,\,\text{and} \,\,\, A^*u=\sigma v,
\end{equation*}
where the vectors $u$ and $v$ are called left-singular and right-singular vectors of $A$ for the singular value $\sigma$.

Assuming that the singular values of $A$ are ordered decreasingly, that is, $\sigma_1\geq \sigma_2\geq\dots$, and let $u_i$ and $v_i$ be the singular vectors for $\sigma_i$. We emphasis that the vectors $u_i$ and $v_i$ are orthonormal sets in $\mathbb{C}^{n_x}$ and $\mathbb{C}^{n_t}$ respectively. Thus, the SVD of the matrix $A$ is given by
\begin{equation}\label{svd}
A=\Sigma_{i=1}^{n_t}\sigma_i u_i\overline{v_i}^T.
\end{equation}

Since the matrix $A$ is composed of a large rank one part $A_c$ and a small part coming from metabolic signal $\Gamma_{ODT}^m$, we can say that $A-A_c$ is "relatively small" with respect to $A$. It is well known that the first term in the SVD expansion of $A$ is the rank one matrix $A_1$ such that $\|A-A_1\|_{op}$ is minimal. Therefore, it is natural to think that $A_c$ is "close" to $A_1$ in some way. But $A_1=\sigma_1 u_1\overline{v_1}^T$ is generally not the same as $A_c$, because as we will see in Appendix \ref{App:AppendixA}, the eigenvectors of the kernels $F_{cc}$, $F_{cm}$ and $F_{mc}$ are generally not orthogonal. Since SVD always gives an orthogonal set of eigenvectors, we conclude that the SVD approach itself does not give the eigenvectors exactly. Nevertheless, we can show that the SVD result is still a good approximation to the true eigenvectors.

\begin{table}
  \centering
  \begin{tabular}{p{3.5cm}|p{2.5cm}|p{2.5cm}|p{2.5cm}}
  \hline
  \hline
  {} & Total signal $\Gamma_{ODT}$ & First term in the SVD of $\Gamma_{ODT}$ & collagen signal $\Gamma_{ODT}^c$\\
  \hline
  Matrix & $A$ & $A_1$ & $A_c$ \\
  \hline
  First singular value & $\sigma_1$ & $\sigma_1$ & $\sigma_c$ \\
  \hline
  First singular vector & $u_1,v_1$ & $u_1,v_1$ & $u_c,v_c$ \\
  \hline
  Other singular values & $\sigma_2>\sigma_3>\dots>\sigma_i$ & $0$ & $0$ \\
  \hline
  \hline
\end{tabular}
  \caption{Singular values and singular vectors.}\label{tableSymbols}
\end{table}

To bridge the gap between the collagen signal $\Gamma_{ODT}^c$ and the first term of SVD expansion of $\Gamma_{ODT}$, we investigate the relationship between their singular values and singular vectors. We note that $\Gamma_{ODT}^c$ has only one nonzero singular value $\sigma_c$, with the corresponding singular vectors $u_c$ and $v_c$.

We claim in the following theorem that the singular value $\sigma_1$ and the corresponding singular vector $u_1$ are good approximations of the singular value $\sigma_c$ and singular vector $u_c$. See Table \ref{tableSymbols} for the notations of their singular values and singular vectors.

\begin{theorem}\label{eigenvector-difference}
Let $\sigma_i$, $u_i$, $v_i$, $A_c$, $u_c$ and $v_c$ be described in Table \ref{tableSymbols}. Assume that the collagen signal dominates, that is,
\begin{equation}\label{def-N}
\frac{\|A-A_c\|_{op}}{\|A_c\|_{op}}=1/N
\end{equation}
for a large $N$. Then there exists a constant $C>0$ such that
\begin{equation*}
\frac{|\sigma_c-\sigma_1|}{\sigma_c}\leq C/N,
\end{equation*}
and
\begin{equation*}
\|u_c-u_1\|_{l^2}\leq C/N.
\end{equation*}

\end{theorem}
\begin{proof}

We define a matrix-valued function
\begin{equation}\label{F-function}
F: s\mapsto (A_c+sN(A-A_c))^*(A_c+sN(A-A_c)).
\end{equation}

Through this construct of $F$, we obtain
\begin{equation*}
F(0)=A_c^*A_c \,\,\textrm{and }\,\, F(\frac{1}{N})=A^*A.
\end{equation*}

Applying Rellich's perturbation theorem on hermitian matrices $F$ (see, for example, \cite{rellich1969perturbation}) to get the following two properties. There exists a set of $n$ analytic functions $\lambda_1(s),\lambda_2(s),\dots$, such that they are all the eigenvalues of $F(s)$. Also, there exists a set of vector-valued analytic functions $u_1(s),u_2(s),\dots$, such that $F(s)u_i(s)=\lambda_i(s)u_i(s)$, and $\langle u_i(s),u_j(s)\rangle=\delta_{ij}$.

In view of the definition of $u_i(s)$ and $\lambda_i(s)$, we show four useful properties,
\begin{align}\label{relation-u-sigma}
\begin{aligned}
u_1(0)&=u_c,& u_1(1/N)&=u_1,\\
\lambda_1(0)&=\sigma_c^2=\|A_c\|_{op}^2,& \lambda_1(1/N)&=\sigma_1^2,
\end{aligned}
\end{align}
where the last property comes from the fact $\lambda_1(1/N)$ is the largest eigenvalue of $F(1/N)=A^*A$ when $N\gg 1$.

The objective is to get upper bounds for $\|u_c-u_1\|_{l^2}$ and $|\sigma_c-\sigma_1|$. Using (\ref{relation-u-sigma}), we have $u_c-u_1=u_1(0)-u_1(1/N)$ and $\sigma_c-\sigma_1=\sqrt{\lambda_1(0)}-\sqrt{\lambda_1(1/N)}$. Since $u_1(s)$ and $\lambda_1(s)$ are analytic, a Taylor expansion at $0$ yields
\begin{equation}\label{estimate_sigma_u}
\begin{split}
\|u_c-u_1\|_{l^2}&=\|\frac{u_1'(0)}{N}\|_{l^2}+O(1/N^{3/2}), \\
|\sigma_c-\sigma_1|&=\frac{\lambda_1'(0)}{2\sqrt{\lambda_1(0)}N}+O(1/N^2).
\end{split}
\end{equation}

The next step is to seek for proper upper bounds for $\lambda_1'(0)$ and $u_1'(0)$.

For the upper bound of $\lambda_1'(0)$, we differentiate $F(s)u_i(s)=\lambda_i(s)u_i(s)$ with respect to $s$ and then take $s=0$ to obtain
\begin{equation}\label{Fu}
F'(0)u_i(0)+F(0)u_i'(0)=\lambda_i(0)u_i'(0)+\lambda_i'(0)u_i(0).
\end{equation}
Since we always have $\|u_i(s)\|_{\ell^2}=1$, a direct calculation shows that $$\langle u_i(s),u_i'(s)\rangle=\frac12\frac{d}{ds}\|u_i(s)\|^2=0.$$ By taking an inner product of both sides of (\ref{Fu}) with $u_i(0)$, we get
\begin{align*}\begin{split}
\lambda_i'(0)&=\lambda_i'(0)\|u_i(0)\|^2_{l^2}\\
&=\langle u_i(0),F'(0)u_i(0)\rangle+\langle u_i(0),F(0)u_i'(0)\rangle\\
&=\langle u_i(0),F'(0)u_i(0)\rangle+\langle F(0)u_i(0),u_i'(0)\rangle\\
&=\langle u_i(0),F'(0)u_i(0)\rangle+\lambda_i(0)\langle u_i(0),u_i'(0)\rangle\\
&=\langle u_i(0),F'(0)u_i(0)\rangle.
\end{split}\end{align*}

Hence, $\lambda_i'(0)$ satisfies $|\lambda_i'(0)|\leq \|F'(0)\|_{op}$. By the definition of $F(s)$, we have $\|F'(0)\|_{op}=N\|A_c^*(A-A_c)+(A-A_c)^*A_c\|\leq 2N\|A_c\|_{op}\|A-A_c\|_{op}$. Replacing $N$ with (\ref{def-N}) yields $|\lambda_i'(0)|\leq 2\|A_c\|^2_{op}$. Therefore, by inserting the expression $\lambda_1(0)$ in (\ref{relation-u-sigma}) into (\ref{estimate_sigma_u}), we get $|\sigma_c-\sigma_1|\leq \frac{\sigma_c}{N}+O(1/N^2)$.

For the upper bound of $u_1'(0)$, we look again at (\ref{Fu}). By taking an inner product with $u_1'(0)$, we immediately obtain
\begin{equation*}
\langle u_1'(0),F'(0)u_1(0)\rangle+\langle u_1'(0), F(0)u_1'(0)\rangle=\lambda_1(0)\|u_1'(0)\|^2_{l^2}.
\end{equation*}

Recall that the matrix $A_c$ is of rank one. So, there exists a positive constant $c$, such that $A_c^*A_c=cu_1(0)u_1^T(0)$, which reads
\begin{equation*}
F(0)u_1'(0)=cu_1(0)(u_1^T(0)u_1'(0))=cu_1(0)\langle u_1(0),u_1'(0)\rangle=0.
\end{equation*}

Therefore, direct calculation shows that $\|u_1'(0)\|_{l^2}\leq \frac{\|F'(0)u_1(0)\|_{l^2}}{\lambda_1(0)}\leq \frac{\|F'(0)\|_{op}}{\|A_c\|^2_{op}}\leq 2$.

The rest of the proof follows by substituting the above bound into (\ref{estimate_sigma_u}), then we have $\|u_c-u_1\|_{l^2}\leq\frac{2}{N}+O(1/N^{3/2})$.
\end{proof}

\textbf{Remark 1.} Theorem \ref{eigenvector-difference} shows that the eigenvector difference of two classes is the order of $\frac{1}{N}$, where $N$ could be seen as the ratio between collagen signal and metabolic signal, so when $N$ is large enough, the difference could be ignored, therefore, it is reasonable to use the eigenvectors of the SVD to approximate the true eigenvectors.

\textbf{Remark 2.} In the proof of Theorem \ref{eigenvector-difference}, we did not use any representation of $A$ and $A_c$, so in a more general case, for any matrix $A=A_c+o(A_c)$ where rank of $A_c$ is 1, the first singular value and first singular vector of $A$ could be used to approximate the singular value and the singular vector of $A_c$.

\subsection{Analysis of obtaining the intensity of metabolic activity}

Recall that our objective is to get the intensity of the metabolic activity after removing the influence of the collagen signal. We have proved that the largest singular value corresponds to the collagen signal, and the following few singular values correspond to the correlation part between collagen signal and metabolic activity, the rest of the singular values contains information related to the metabolic activity.

Let $T$ be the set of these "rest" singular values. In practice, we only know the total signal $\Gamma_{ODT}(x, t)$ (or the matrix $A$). By performing a SVD for $\Gamma_{ODT}(x, t)$, we take the terms only corresponding to the singular values in $T$ in the SVD expansion. The next problem is to reconstruct the intensity of the particle movements of metabolic activity. In our numerical experiments, we observe that the sum $\sum_{i\in T} \sigma_{i}^2|u_{i}(x_j)|^2$ gives a very good approximation to the intensity of metabolic activity at the pixel $x_j$. We will explain why it works.

Physically, we expect the metabolic activity signal to be centered around 0, so in each pixel $x_j$, the norm $\|A_{m}(x_j,t)\|_{\ell^2}^2$ could be seen as the standard deviation of the metabolic signal, which could represent the intensity of metabolic activities in pixel $x_j$. However, in our model the eigenvectors are not orthogonal (this statement may be justified by arguing as in Appendix \ref{App:AppendixA}). Thus when using a SVD, we do not get the exact "pure" metabolic activity signal $A_m$, but only an approximation, which we denote by $A_{m_1}$. We first give an interpretation that $\sum_{i\in T} \sigma_{i}^2|u_{i}(x_j)|^2$ could be written as a $\ell^2$ norm of the matrix $A_{m_1}$.

\begin{theorem}\label{mth-intensity}
Let $A$ be the matrix after the discretization of $\Gamma_{ODT}(x, t)$ with respect to $x$ and $t$, such that the $j$-th row of $A$ corresponds to the pixel $x_j$, and the $k$-th column of $A$ corresponds to the time $t_k$. Let $T$ be a subset of singular values of $A$, and $A_{m_1}$ be the result of taking only the singular values in $T$ from $A$. Then for any pixel $x_j$, we have
\begin{equation}\label{meq-intensity}
\sum_{i\in T} \sigma_i^2|u_i(x_j)|^2=\sum_{k}|A_{m_1}(x_j,t_k)|^2.
\end{equation}

\end{theorem}
\begin{proof}

We apply the SVD algorithm to the matrix $A$ to get $A=USV^*$, where $U=(u_1,u_2,\dots)$, $V=(v_1,v_2,\dots)$ are unitary matrices, and $S$ is a diagonal matrix containing the singular values of $A$.

We construct a new diagonal matrix $S_T$, which is obtained from $S$ by keeping all the singular values in $T$, but changing everything else to zero. By the definition of $A_{m_1}$, we readily derive $A_{m_1}=US_TV^*$.

Note that $\sigma_i u_i(x_j)$ is the element at row $j$, column $i$ of the matrix $US$. By the construction of $S_T$, we know $\sigma_i u_i(x_j)$ is the element at row $j$, column $i$ of the matrix $US_T$ for every $i\in T$. Therefore, the sum $\sum_{i\in T} \sigma_i^2|u_i(x_j)|^2$ is equal to the square-sum of the $j$-th row of the matrix $US_T$, which gives
\begin{equation}\label{ussigma}
\|US_T(x_j,\cdot)\|_{\ell^2}^2=\sum_{i\in T}\sigma_i^2|u_i(x_j)|^2.
\end{equation}

On the other hand, the relation $A_{m_1}=(US_T)V^*$ means that for each $x_j$, $A_{m_1}(x_j,\cdot)=(US_T)(x_j,\cdot)V^*$.

A direct calculation from the definition of $\ell^2$ norm of vectors shows that
\begin{equation*}
\sum_{k}|A_{m_1}(x_j,t_k)|^2=\|A_{m_1}(x_j,\cdot)\|_{\ell^2}^2=A_{m_1}(x_j,\cdot)A_{m_1}(x_j,\cdot)^*.
\end{equation*}

Using $V^*V=I$ and substituting $(US_T)(x_j,\cdot)V^*$ for $A_{m_1}(x_j,\cdot)$ yields

\begin{equation}\label{sum_jforA}
\sum_{k}|A_{m_1}(x_j,t_k)|^2=US_T(x_j,\cdot)(US_T(x_j,\cdot))^*=\|US_T(x_j,\cdot)\|_{\ell^2}^2.
\end{equation}

Combining (\ref{ussigma}) and (\ref{sum_jforA}) completes the proof.
\end{proof}

Then let us look at the $\ell^2$ norm of the difference between the two matrices $A_m$ and $A_{m_1}$. Proceeding as in the proof of Theorem \ref{eigenvector-difference}, we can estimate $\|A_{m}-A_{m_1}\|$.  When $N$ in (\ref{def-N}) is large enough, it is reasonable to approximate $A_m$ by $A_{m_1}$. This fact enables us to say that $\|A_{m_1}(x_j,t)\|_{\ell^2}^2\approx \|A_{m}(x_j,t)\|_{\ell^2}^2$ for each pixel $x_j$.

Therefore, we conclude that $\sum_{i\in T} \sigma_i^2|(u_i)_j|^2$ over the set $T$ of "rest" singular values is indeed a good approximation for the metabolic activity intensity.

\section{Numerical experiments}

In this section we model the forward measurements of our problem. Using the SVD decomposition we filter out the signal, finally obtaining images of the hidden weak sources.

\subsection{ Forward problem measurements }

To simulate the signal measurements using formula \eqref{smallareamodel}, we only need to simulate the density function $p(x,z,t)$ of the media to be illuminated. For each pixel $x$, there are two types of superimposed media. One is the collagen media characterized for having a strong signal and slow movement. The second medium is the metabolic activity, that has a fast movement relative to the time samples. According to \cite{apelian2016dynamic}, the collagen signal intensity is around 100 times stronger than the metabolic one.

Given these properties, both media are modeled differently. The collagen particles are simulated as an extended random medium on $z$ that displaces slowly on time; see \cite{klimevs2002correlation}. For each pixel $x$, an independent one-dimensional random medium $r_x(\cdot)$ is generated, and then $p(x,z,t) = r_x(z + tv)$ with $v$ being the constant movement velocity. The metabolic activity is simulated as an uniform white noise, whose intensity represents its magnitude. Background or instrumental noise is added everywhere in a similar fashion, but with smaller intensity.

After the medium is simulated, formula \eqref{smallareamodel} is applied to reproduce the measured signal, where for integration purposes, the broadband light is approximated by Dirac deltas in certain frequencies. All the model parameters are set such that we obtain similar measurements to the ones obtained in \cite{apelian2016dynamic}. In Figure \ref{fig:signal} we can see, for a single pixel, the simulated signal as a function of time.

\begin{figure}[h!]
\centering
	\begin{minipage}[c]{.28\textwidth}
		\centering
		\ \ \  \scriptsize Total Signal
	\end{minipage}
	\hspace{.5cm}
	\begin{minipage}[c]{.28\textwidth}
		\centering
		\ \  \ \scriptsize Collagen  Signal
	\end{minipage}
	\hspace{.5cm}
	\begin{minipage}[c]{.28\textwidth}
		\centering
		\ \  \ \scriptsize Metabolic Signal
	\end{minipage}
	\\
	\begin{minipage}[c]{.28\textwidth}
		\centering
		\includegraphics[width=\linewidth]{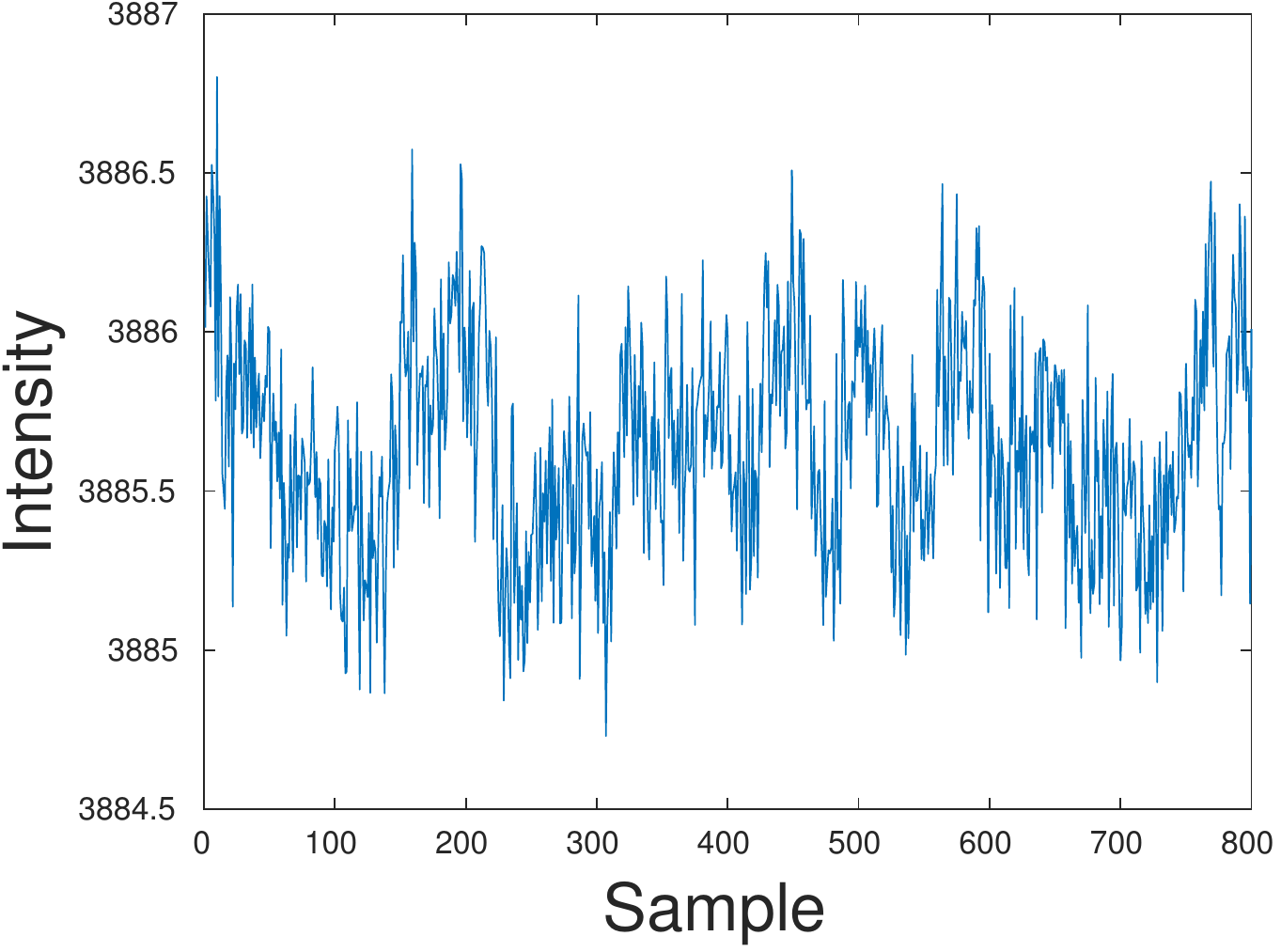}
	\end{minipage}
	\hspace{.5cm}
	\begin{minipage}[c]{.28\textwidth}
		\centering
		\includegraphics[width=\linewidth]{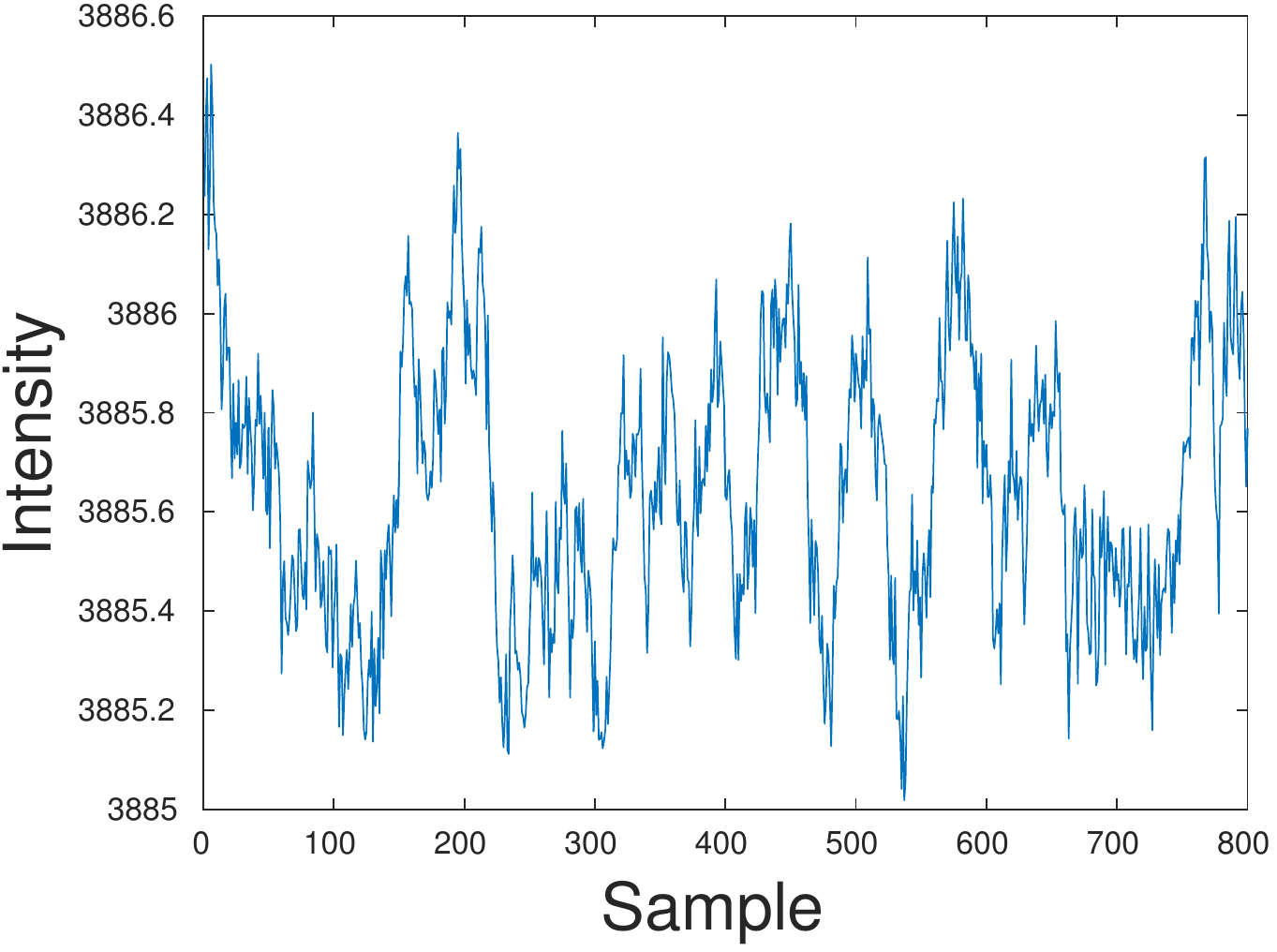}
	\end{minipage}
	\hspace{.5cm}
	\begin{minipage}[c]{.28\textwidth}
		\centering
		\includegraphics[width=\linewidth]{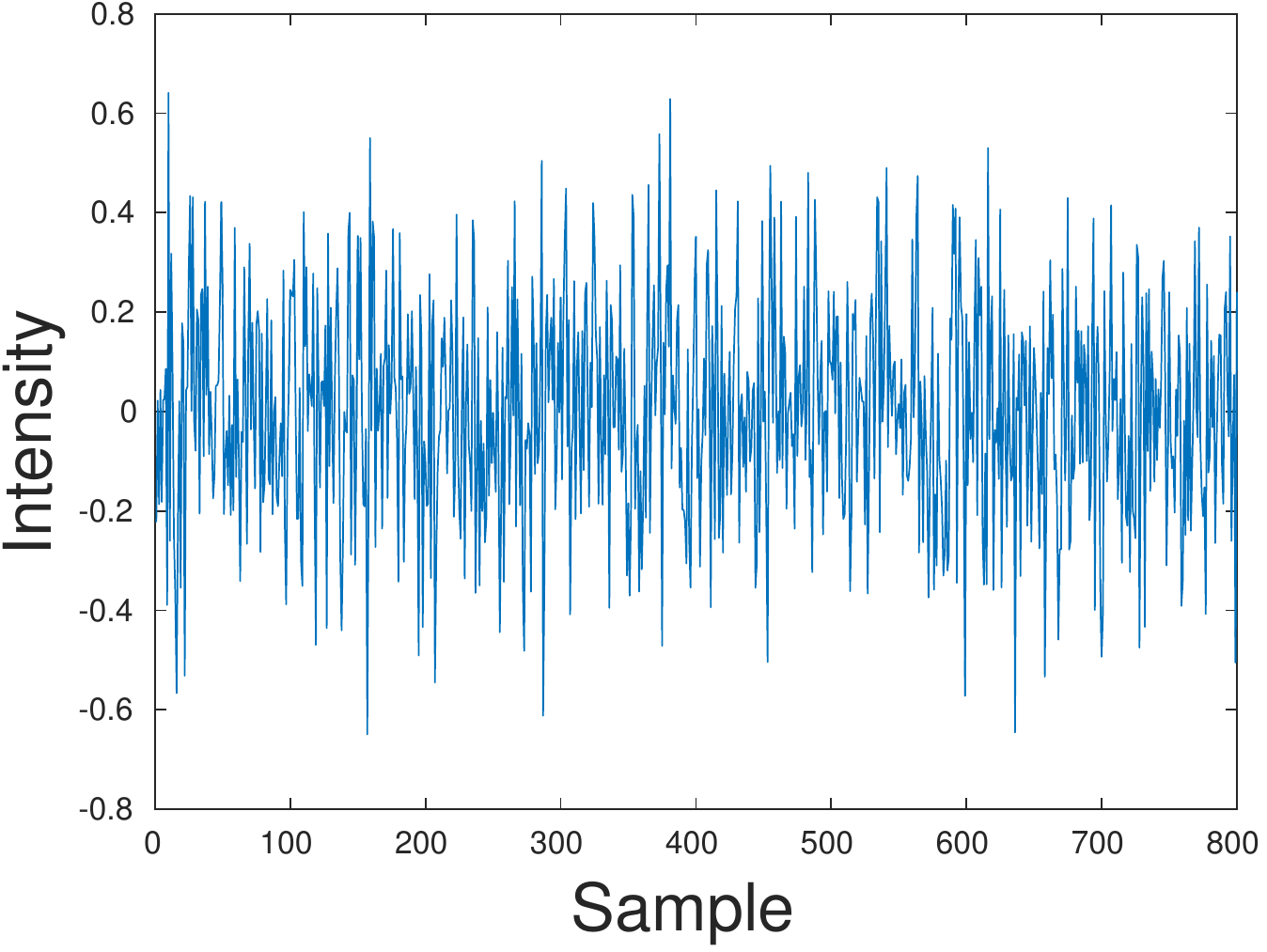}
	\end{minipage}
	\caption{On  the left we can see the total signal measured at a fixed pixel. If decomposed into the one corresponding to the collagen structures and metabolic signal, we obtain the other two images.}
	\label{fig:signal}
\end{figure}

In the following, we consider a two-dimensional 21x21 grid of pixels. The collagen signal, albeit being generated by an independent random media, has the same parameters everywhere, thus sharing a similar behavior. In Figure \ref{fig:metabolicMap}, we present the considered metabolic activity intensity map and two snapshots at different times of the total signal.

\begin{figure}[h!]
	\centering
	\begin{minipage}[c]{.28\textwidth}
		\centering
		 \scriptsize Metabolic activity map \ \ \
	\end{minipage}
	\hspace{.5cm}
	\begin{minipage}[c]{.28\textwidth}
		\centering
		\scriptsize Snapshot of measurements \ \ \
	\end{minipage}
	\hspace{.5cm}
	\begin{minipage}[c]{.28\textwidth}
		\centering
		\scriptsize Snapshot of measurements \ \ \
	\end{minipage}
	\\
	\begin{minipage}[c]{.28\textwidth}
		\centering
		\includegraphics[width=\linewidth]{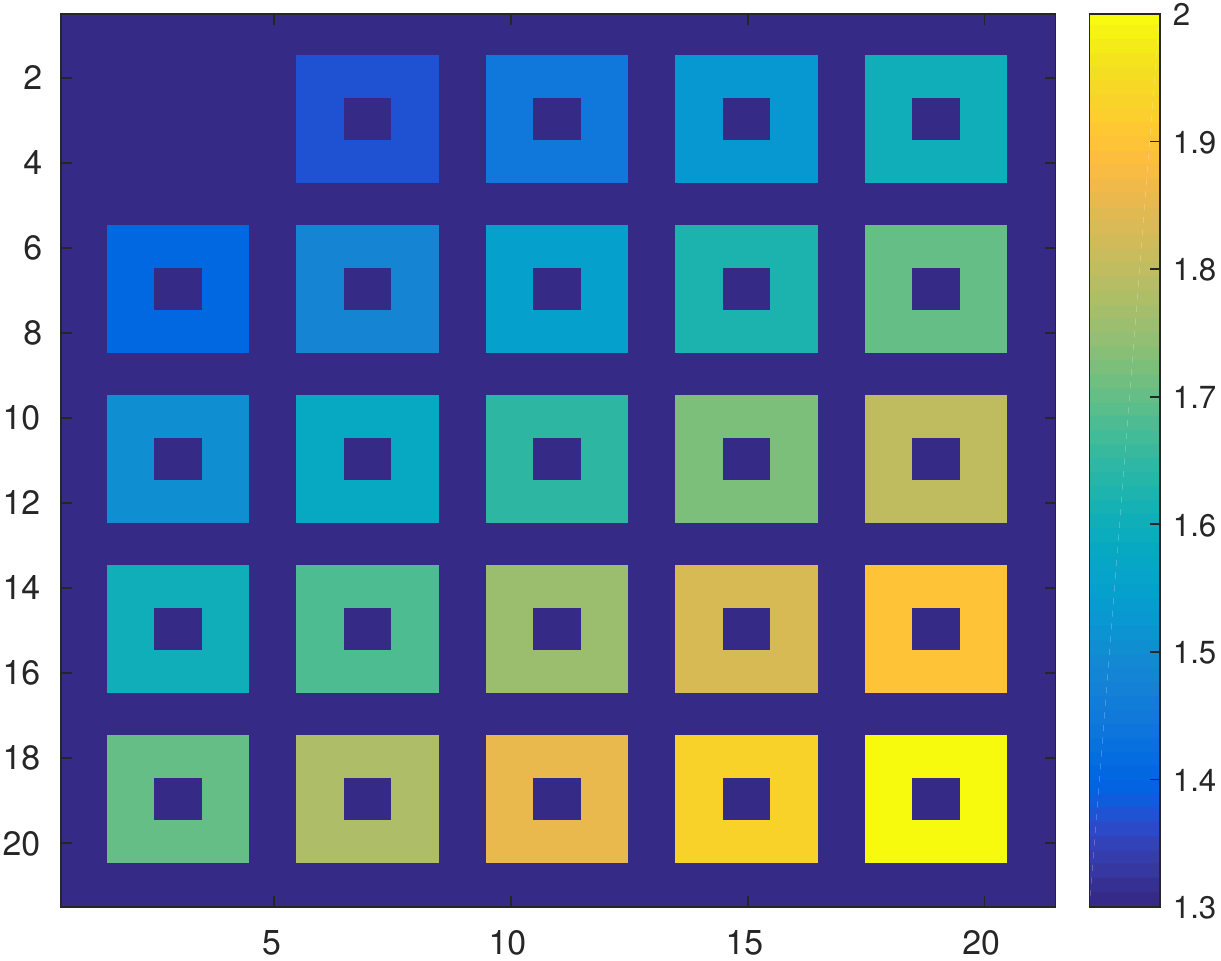}
	\end{minipage}
	\hspace{.5cm}
	\begin{minipage}[c]{.28\textwidth}
		\centering
		\includegraphics[width=\linewidth]{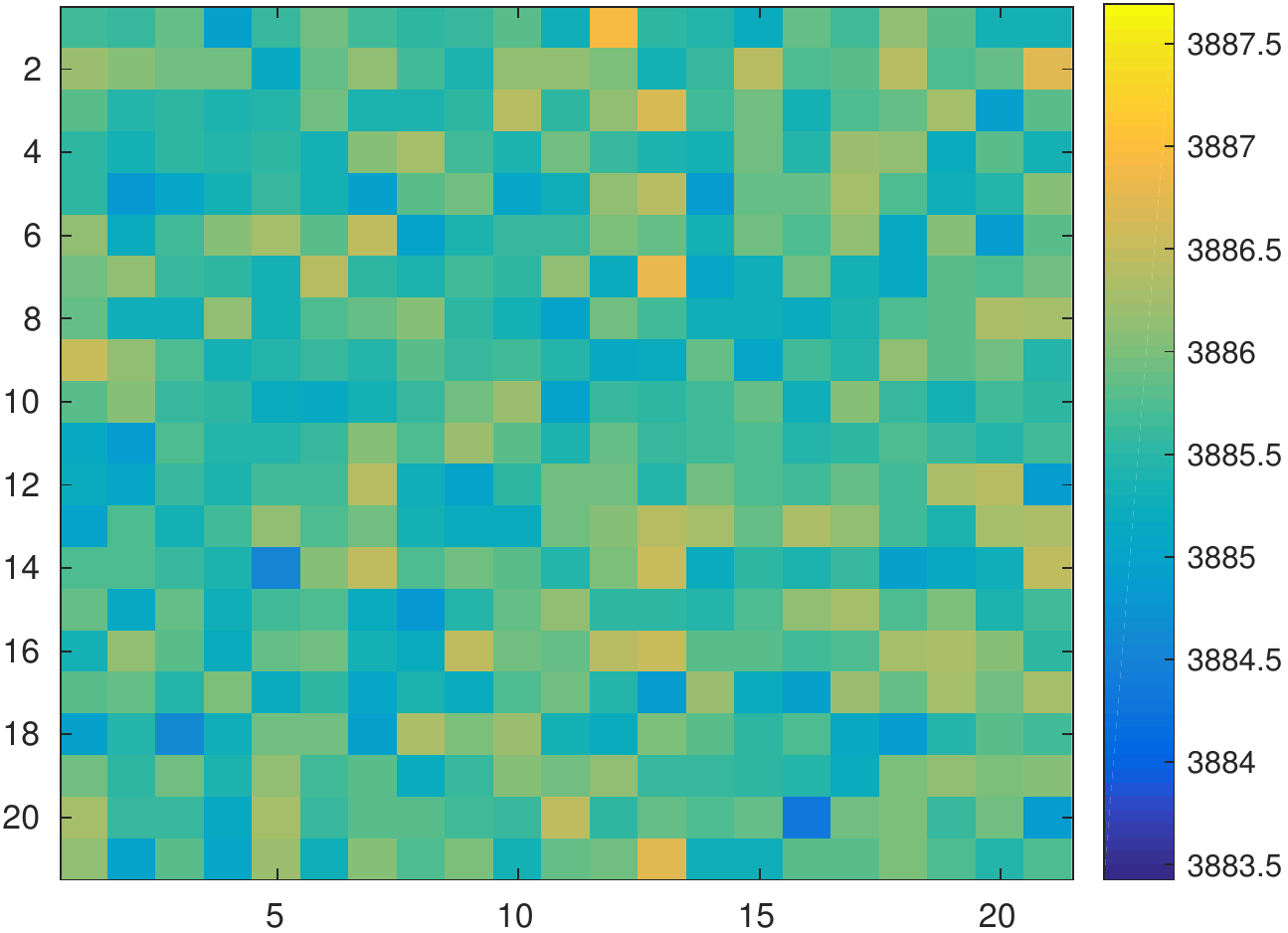}
	\end{minipage}
	\hspace{.5cm}
	\begin{minipage}[c]{.28\textwidth}
		\centering
		\includegraphics[width=\linewidth]{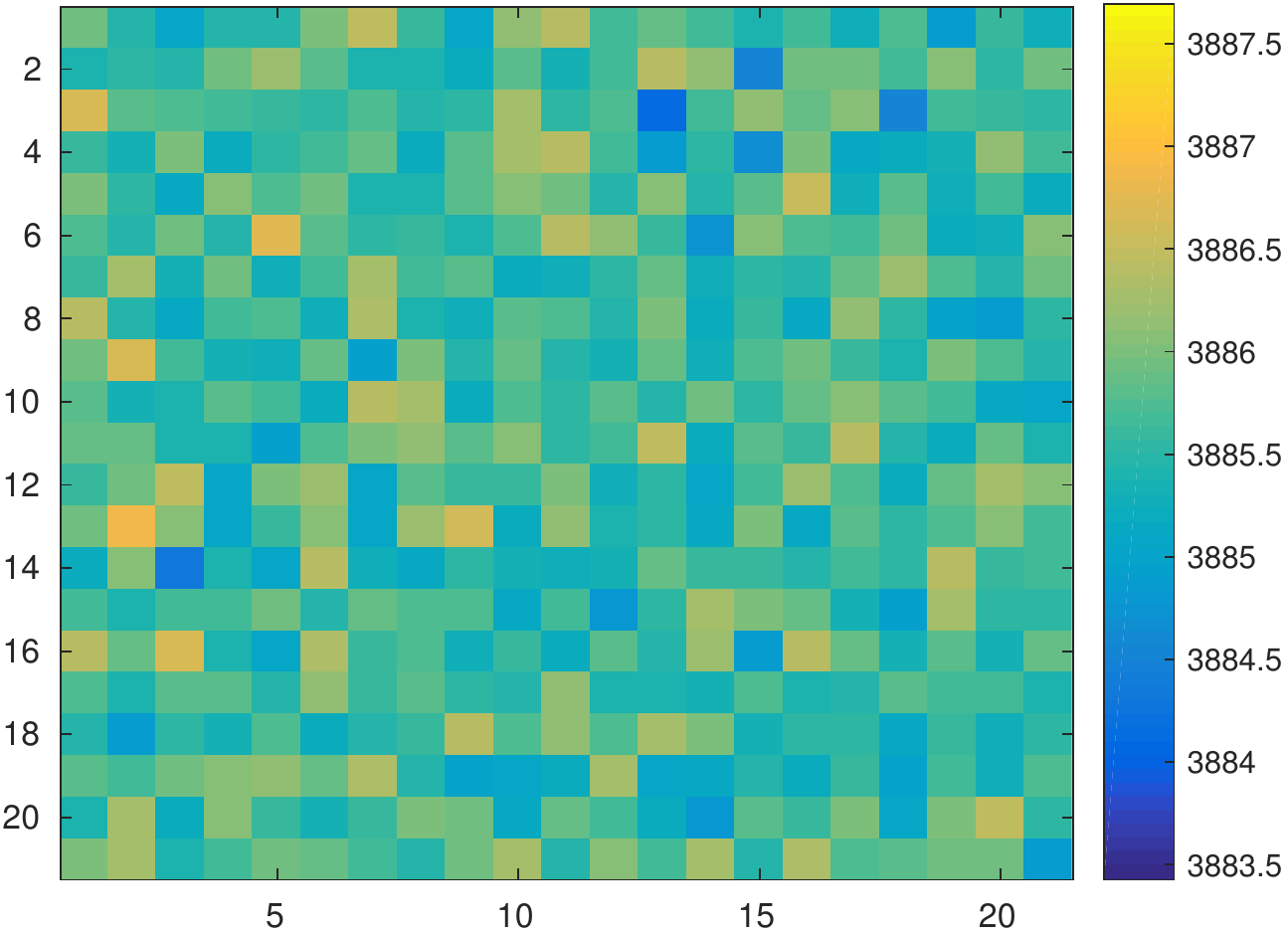}
	\end{minipage}
	\caption{On  the left we can see the considered metabolic map, it describes the intensity of the metabolic signal presented in Figure \ref{fig:signal}. The other two images correspond to raw sampling of the media at different times.}
	\label{fig:metabolicMap}
\end{figure}

\subsection{SVD of the measurements}

To use the singular value decomposition on the signal, we reshape the raw data $\Gamma_{ODT}(x,t)$ under a Casorati matrix form, where the two-dimensional pixels on the $x$ variable are rearranged as a one-dimensional variable, and hence the total signal is written as a matrix $A$ where each dimension corresponds respectively to the space and time variables. The total signal consists on the addition of the metabolic and collagen signals, namely $A = A_m + A_c$. Our objective is to recover the spatial information of the metabolic signal $A_m$.
%yields
%\begin{equation*} A = \sum_{i = 1}^{\text{rank}(A)} \sigma_i u_i v_i^t, \end{equation*}
%where $\left\{ u_i \right\}_i$ and  $\left\{ v_i \right\}_i$ are orthonormal bases of the space and time spaces, also called singular space/time vectors.

We apply the SVD decomposition (\ref{svd}) over the total signal $A$, where the dimension of each space corresponds to the amount of pixels of the image and the time samples respectively. Each space vector $\left\{ u_i \right\}$ point out which pixels are participating in the $i$th singular value. To obtain an image of the pixels participating in a particular subset of singular values $ T \subset \mathds{N}$, we use the following formula (see Section 4.2 for why it works):
\begin{equation} \label{eq:reconstructFormula} I(j) = \sqrt{ \sum_{i \in T} \sigma_i^2 u_i(j)^2 }, \end{equation}
where the indices $j$ are for indexing the image's pixels. When the signal has mean 0, formula (\ref{eq:reconstructFormula}) corresponds to the standard deviation that was already considered as an imaging formula in \cite{apelian2016dynamic}. \hid{ Another way to write formula \ref{eq:reconstructFormula} is $ I = \sqrt{ \text{diag}(A A^t ) } $, where diag is a function that given a matrix, returns the diagonal terms as a vector.}

In Figure \ref{fig:spaceVectors}, we can see an image of each space vector $\left\{ u_i \right\}$ ordered by their associated singular value, these vectors correspond to the decomposition of the total signal $A$. The other two pictures on the right of it, correspond to the singular space vectors but for each unmixed signal $A_c$ and $A_m$, separately. As it can be seen, the spatial vectors of both signals get mixed in the total signal, but the metabolic activity ones get embedded in a clustered fashion, although there is a distortion of these vectors, this is unavoidable given the nature of the SVD.

\begin{figure}[h!]
	\centering
	\begin{minipage}[c]{.28\textwidth}
		\centering
	\ \ \	\scriptsize $A$ singular space-vectors
	\end{minipage}
	\hspace{.5cm}
	\begin{minipage}[c]{.28\textwidth}
		\centering
	\ \ \	\scriptsize $A_c$ singular space-vectors
	\end{minipage}
	\hspace{.5cm}
	\begin{minipage}[c]{.28\textwidth}
		\centering
	\ \	\scriptsize $A_m$ singular space-vectors
	\end{minipage}
	\\
	\begin{minipage}[c]{.28\textwidth}
		\centering
		\includegraphics[width=\linewidth]{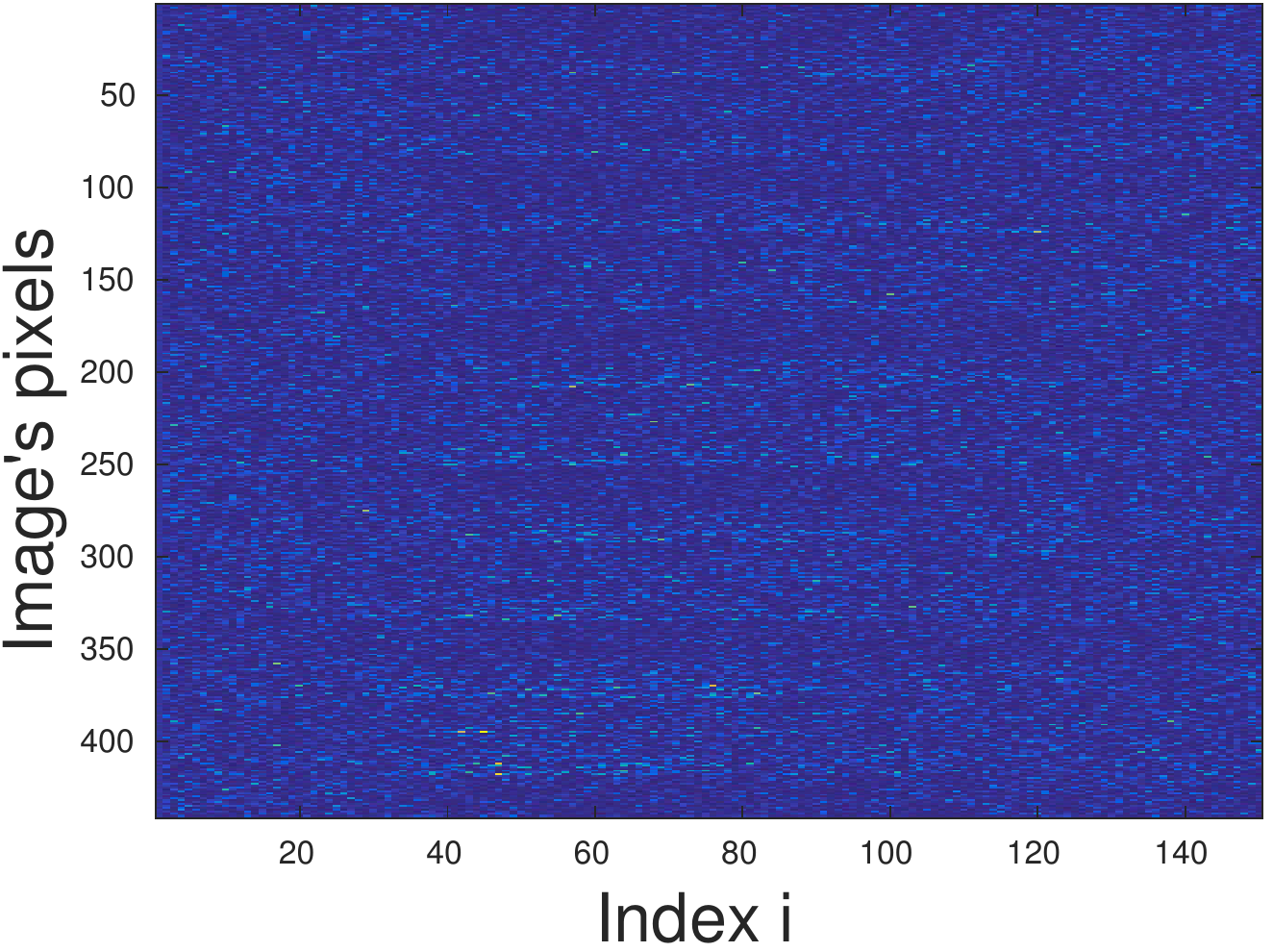}
	\end{minipage}
	\hspace{.5cm}
	\begin{minipage}[c]{.28\textwidth}
		\centering
		\includegraphics[width=\linewidth]{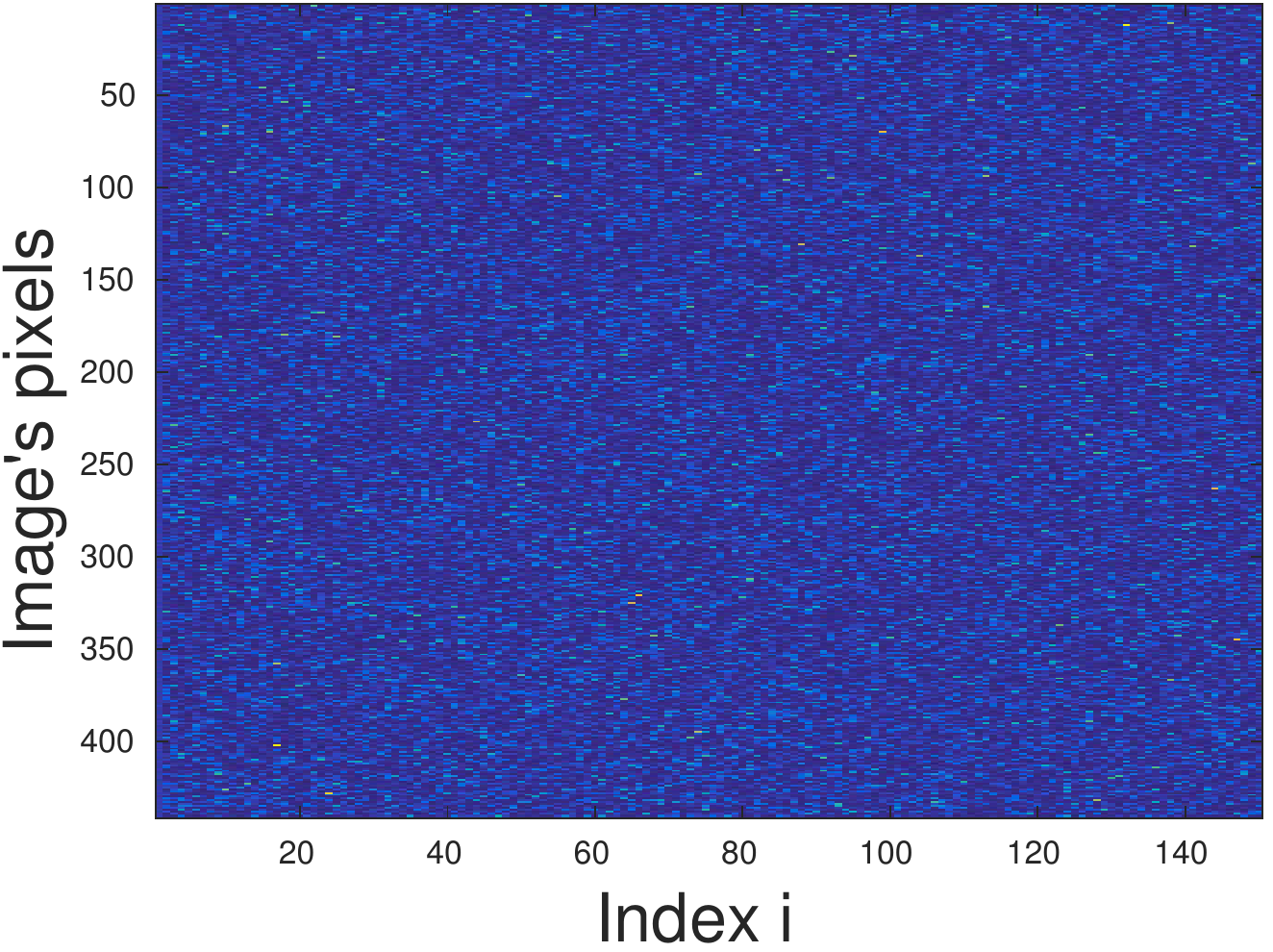}
	\end{minipage}
	\hspace{.5cm}
	\begin{minipage}[c]{.28\textwidth}
		\centering
		\includegraphics[width=\linewidth]{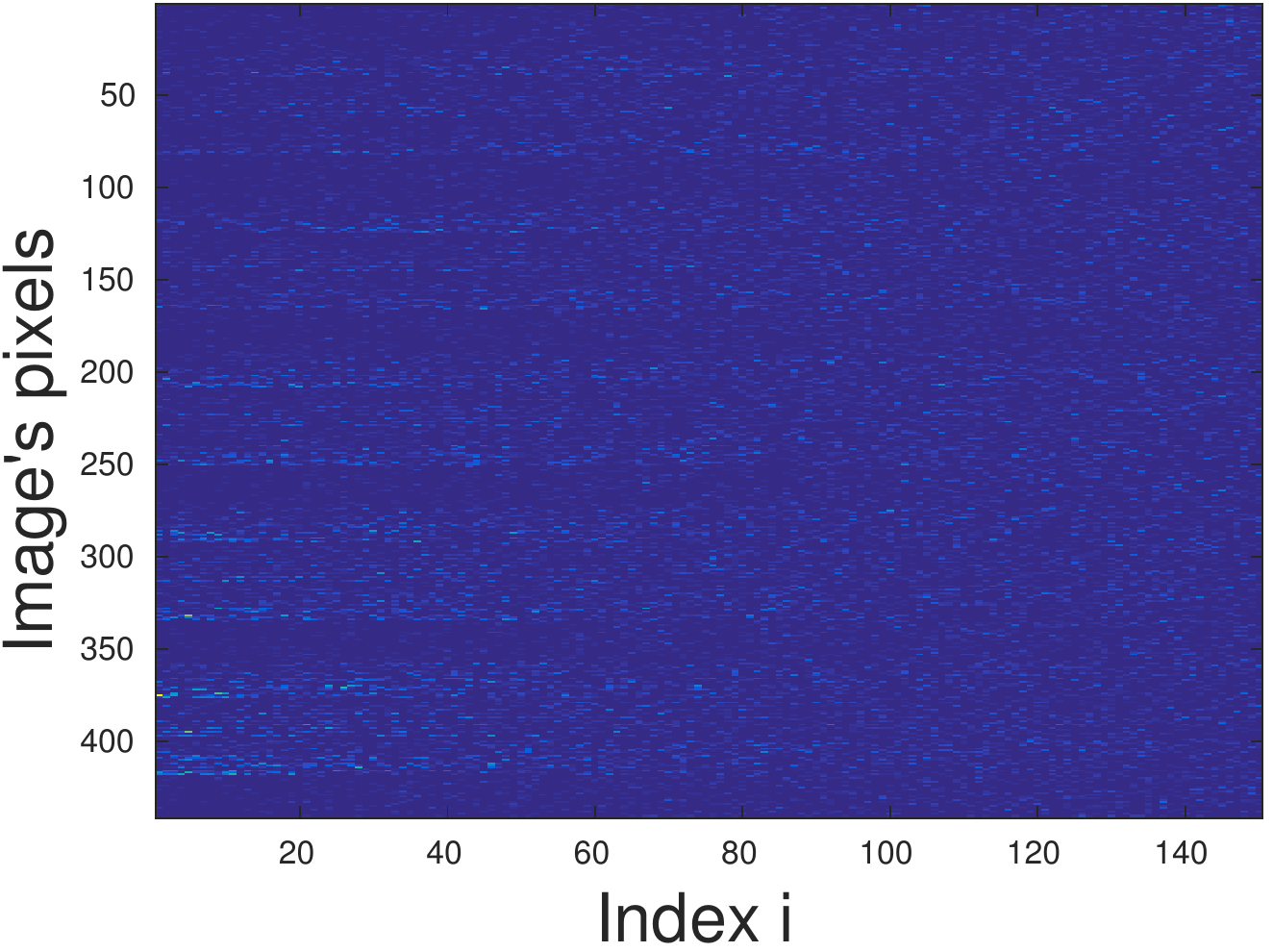}
	\end{minipage}
	\caption{On  the left we can see the singular space vectors of the total signal, ordered by their singular value index and cropped up to the 150th one. On the right we can see the singular space vectors of the decomposed signals: the collagen and the metabolic ones.}
	\label{fig:spaceVectors}
\end{figure}

The location of the spatial vectors is related to their respective singular values, that are presented in Figure \ref{fig:singularValueDecay}. It is observed, that the moment in which the spatial vectors of the total signal start to look like the ones from the metabolic activity, is close to the moment in which the singular values from the metabolic activity get close to those in the total signal. In a mathematical way, we say that the index $j \in \mathds{N}$ in which the spatial vectors $u_i$ start to resemble those of the metabolic activity, corresponds to \begin{equation*} j = \text{argmin} \{ \sigma_j(A) < \sigma_1(A_m) \} - k, \qquad \text{with } k \text{ small.}\end{equation*} In practice, for the tested examples (up to 24x24 grid of pixels, and 500 to 1000 time samples) $k\approx10$ achieve the best results.

The clustered behavior of the singular vectors arise from the model itself, as it generates fast decaying singular values for the collagen signal, whereas the metabolic singular values decay in a more slow fashion. Hence, it is possible to assign an interval of the total signal space vectors as an approximation to the metabolic activity $A_m$.

\begin{figure}[h!]
	\centering
		\includegraphics[width=0.8\linewidth]{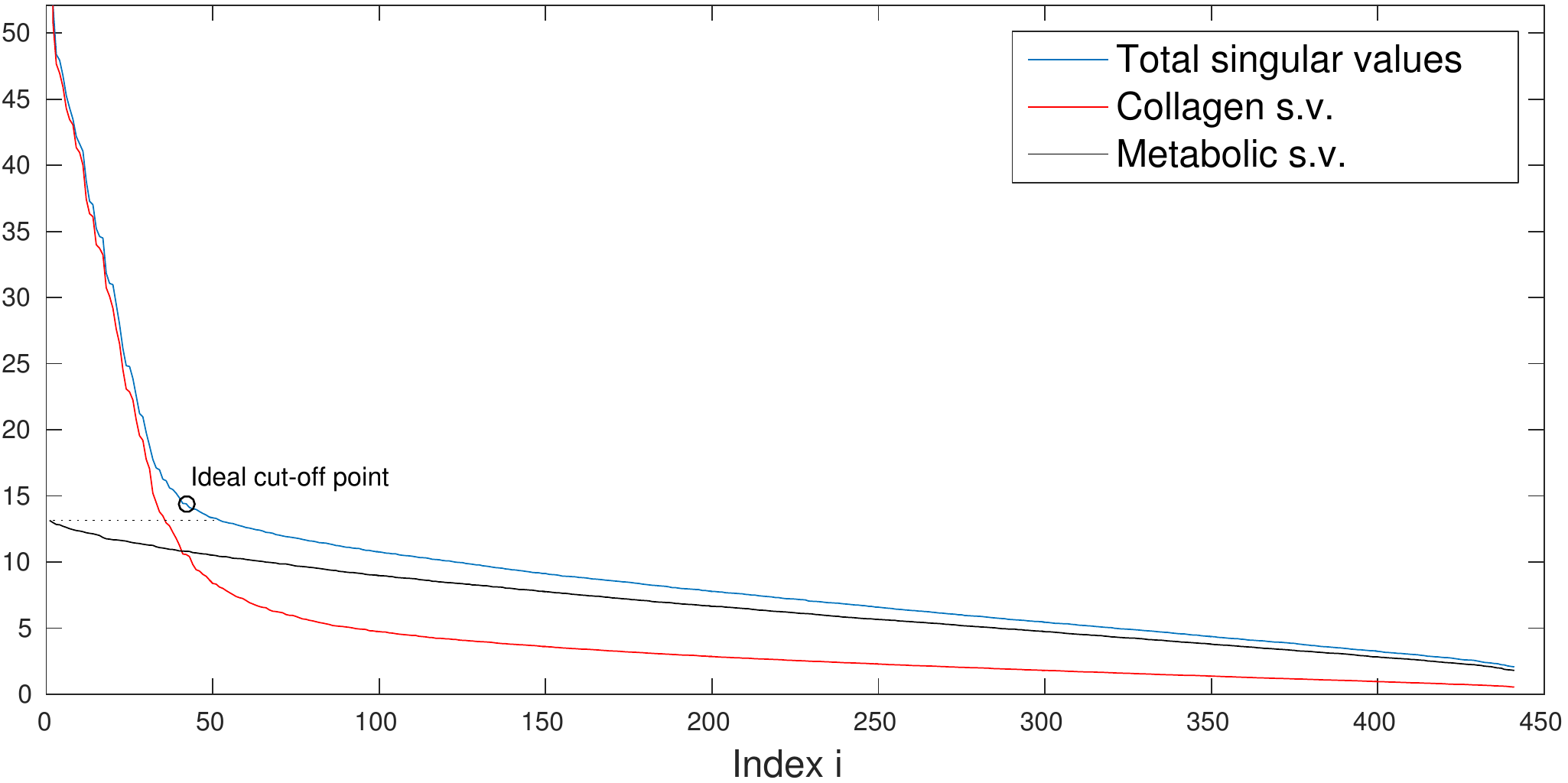}
	\caption{ Singular values for the signals. The circle represents the optimal starting index $j$ at which we should consider the singular space-vectors of the total signal to contain mostly information on the singular space-vectors of the metabolic activity. The first singular value of the total signal and the collagen signal is outside the plot, with an approximate value of $2.3 \times 10^6$.}
	\label{fig:singularValueDecay}
\end{figure}

\subsection{ Selection of cut-off singular value }

\label{subsec:cut-off}

The before mentioned criteria to choose an adequate interval of singular-space vectors to apply the imaging formula (\ref{eq:reconstructFormula}) is not possible in practice, as we have no a priori information on where the metabolic singular values $\sigma_i(A_m)$ lie. Since the idea is to consider an interval of singular space vectors, the first and last elements must be defined. The length of the interval corresponds to the range of the matrix $A_m$, with some added terms coming from the matrix $A_c$. This can be left as a free parameter to be decided by the controller. As a general guideline, it corresponds to the quantity of pixels in which it is expected to find the metabolic activity.

For the considered first singular space vector, also called cut-off one, there is a criteria that arises from the model. Given the differences between the metabolic and collagen signal, the latter in the time variable has some regularity and self correlation. This characteristic is transferred to the first singular time-vectors. In Figure \ref{fig:timeVectors} we can see plots of these time-vectors for each signal.

\begin{figure}[h!]
	\centering
	\begin{minipage}[c]{.28\textwidth}
		\centering
		\tiny Total signal singular time-vector
	\end{minipage}
	\hspace{.5cm}
	\begin{minipage}[c]{.28\textwidth}
		\centering
		\tiny Collagen signal singular time-vector
	\end{minipage}
	\hspace{.5cm}
	\begin{minipage}[c]{.28\textwidth}
		\centering
		\tiny Metabolic signal singular time-vector
	\end{minipage}
	\\
	\begin{minipage}[c]{.28\textwidth}
		\centering
		\includegraphics[width=\linewidth]{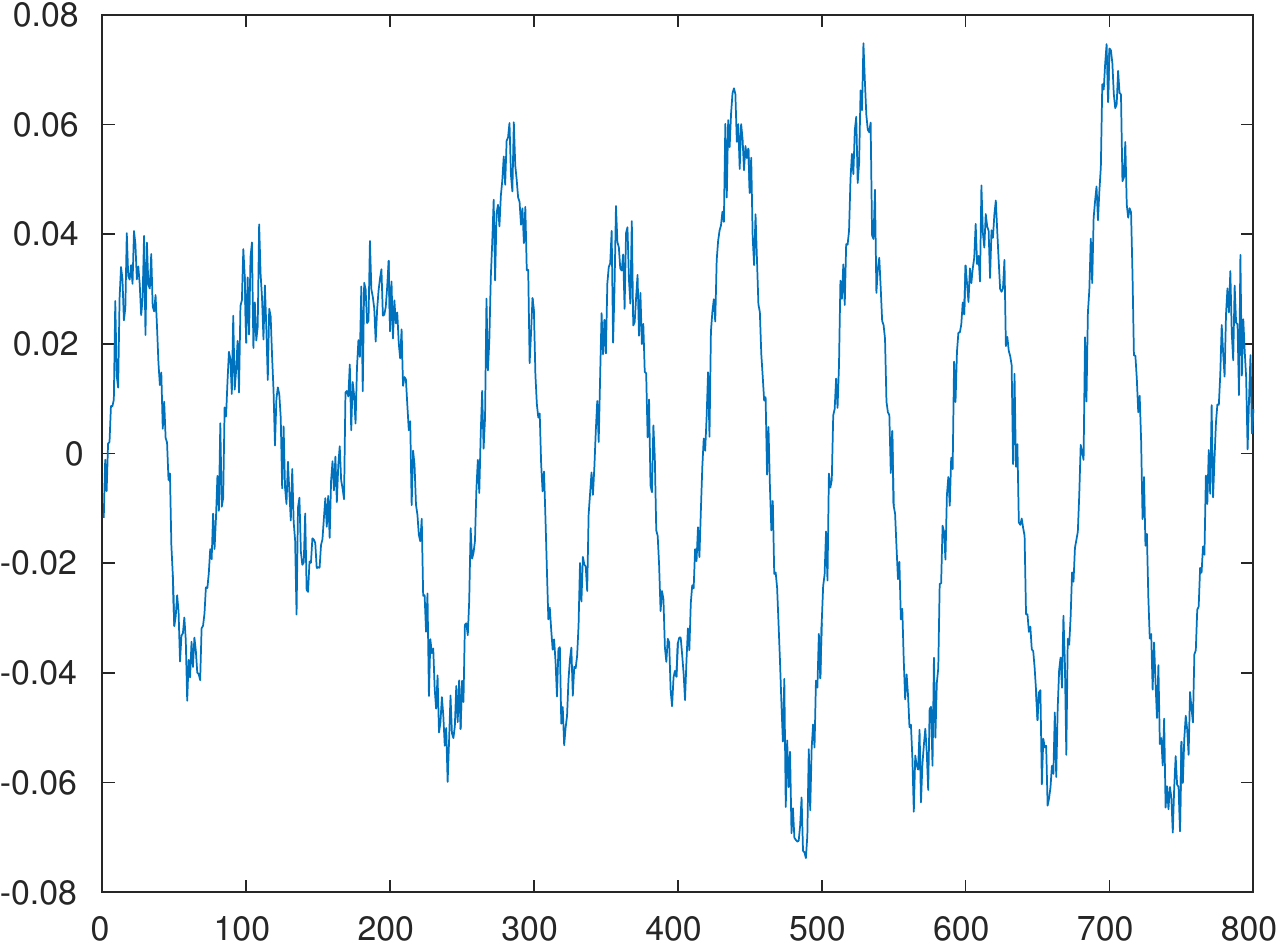}
	\end{minipage}
	\hspace{.5cm}
	\begin{minipage}[c]{.28\textwidth}
		\centering
		\includegraphics[width=\linewidth]{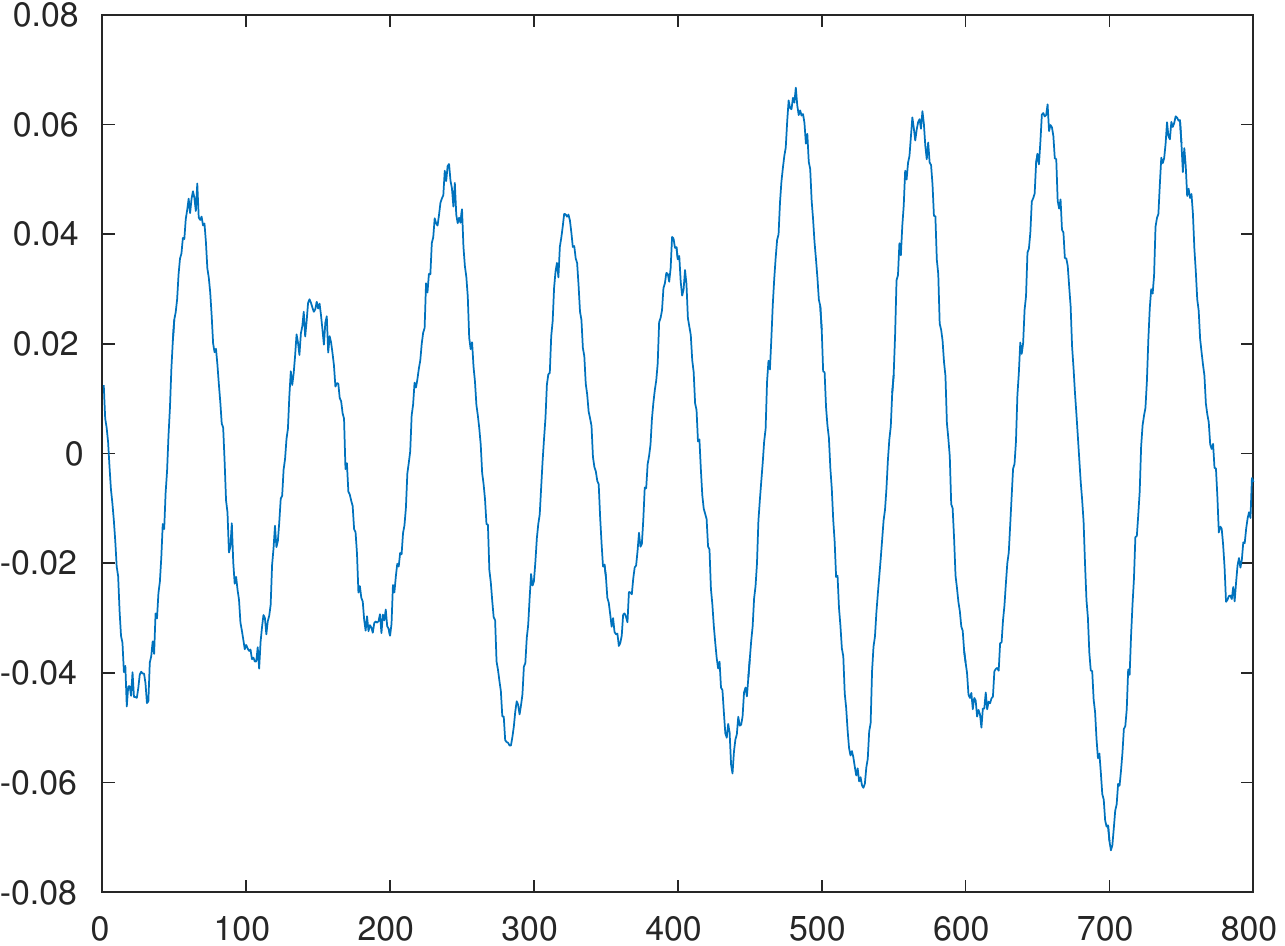}
	\end{minipage}
	\hspace{.5cm}
	\begin{minipage}[c]{.28\textwidth}
		\centering
		\includegraphics[width=\linewidth]{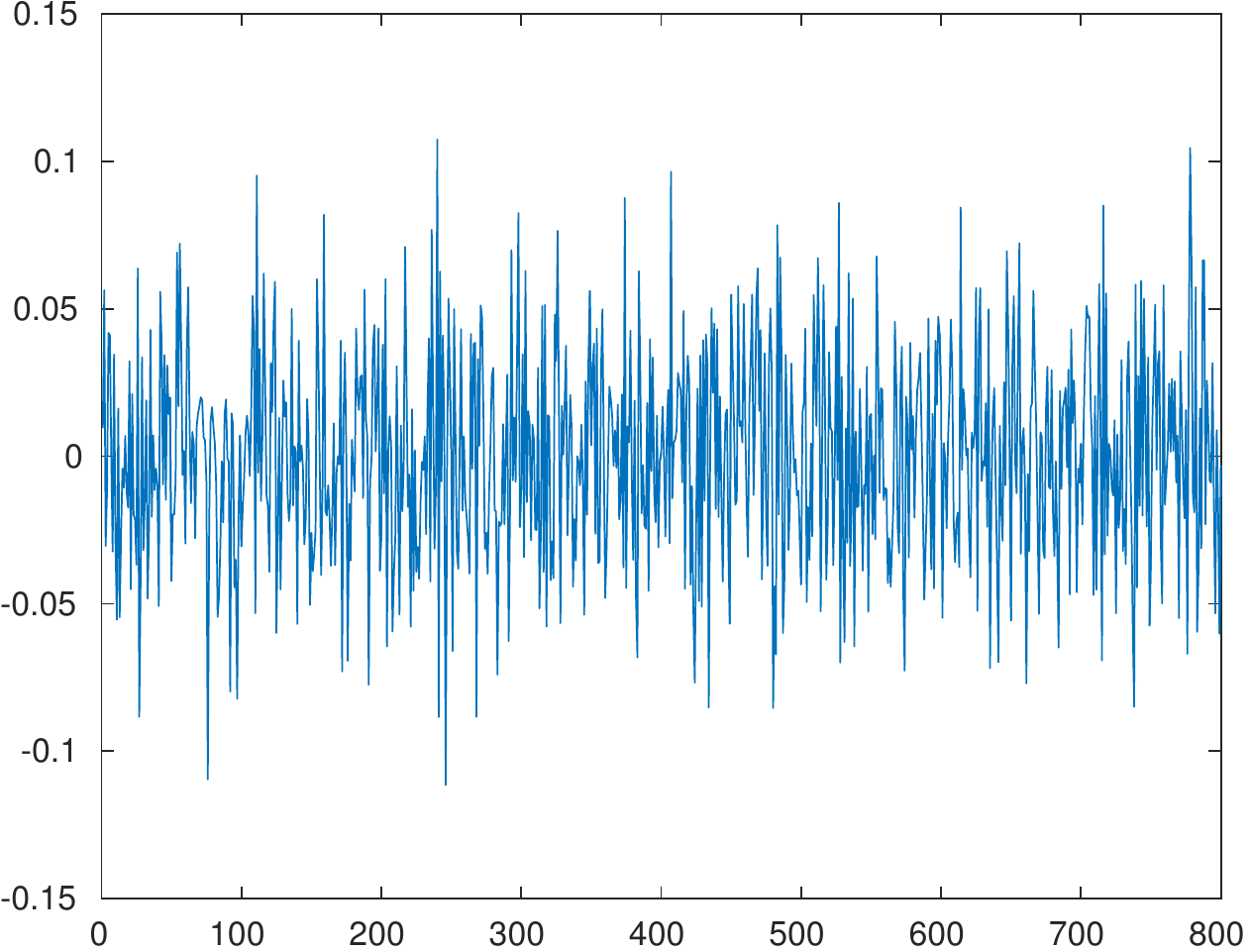}
	\end{minipage}
	\caption{ Plots of the singular time-vectors for each signal, at the second singular value. The collagen time vectors are more regular and correlated compared to the metabolic signal, albeit this property is gradually loosed as we augment the index of the time vectors. Since the SVD of the total signal is dominated by the collagen signal, its time vectors inherit the same property.}
	\label{fig:timeVectors}
\end{figure}

Our proposed technique to decide the cut-off singular space vector consists in measuring the regularity of the time vectors using the total variation semi-norm, the smaller the value the more regular. In the case of a discrete signal, the total variation can be stated as
\begin{equation*}
\left|\ v\ \right|_{\small \text{TV}} = \sum_{i=1}^{N-1} |v(i+1) - v(i)|.
\end{equation*}
Applying the total variation norm to the total signal singular time-vectors $v_i$, we can see that the regularity drops until arriving to, in mean, a slowly increasing plateau. To find it, in an operator free way, it is possible to fit a 2 piece continuous quadratic spline in the total variation plot, and define the cut-off singular value as the point $j$ in which the spline changes. This $l$ is a good approximation for the first singular value of the metabolic activity, meaning that $\sigma_l \approx \sigma_1(A_m)$; see Figure \ref{fig:cutoffTechnique}. Keep in mind that this considered method does not make use of a priori information.

\begin{figure}[h!]
	\centering
	\includegraphics[width=0.8\linewidth]{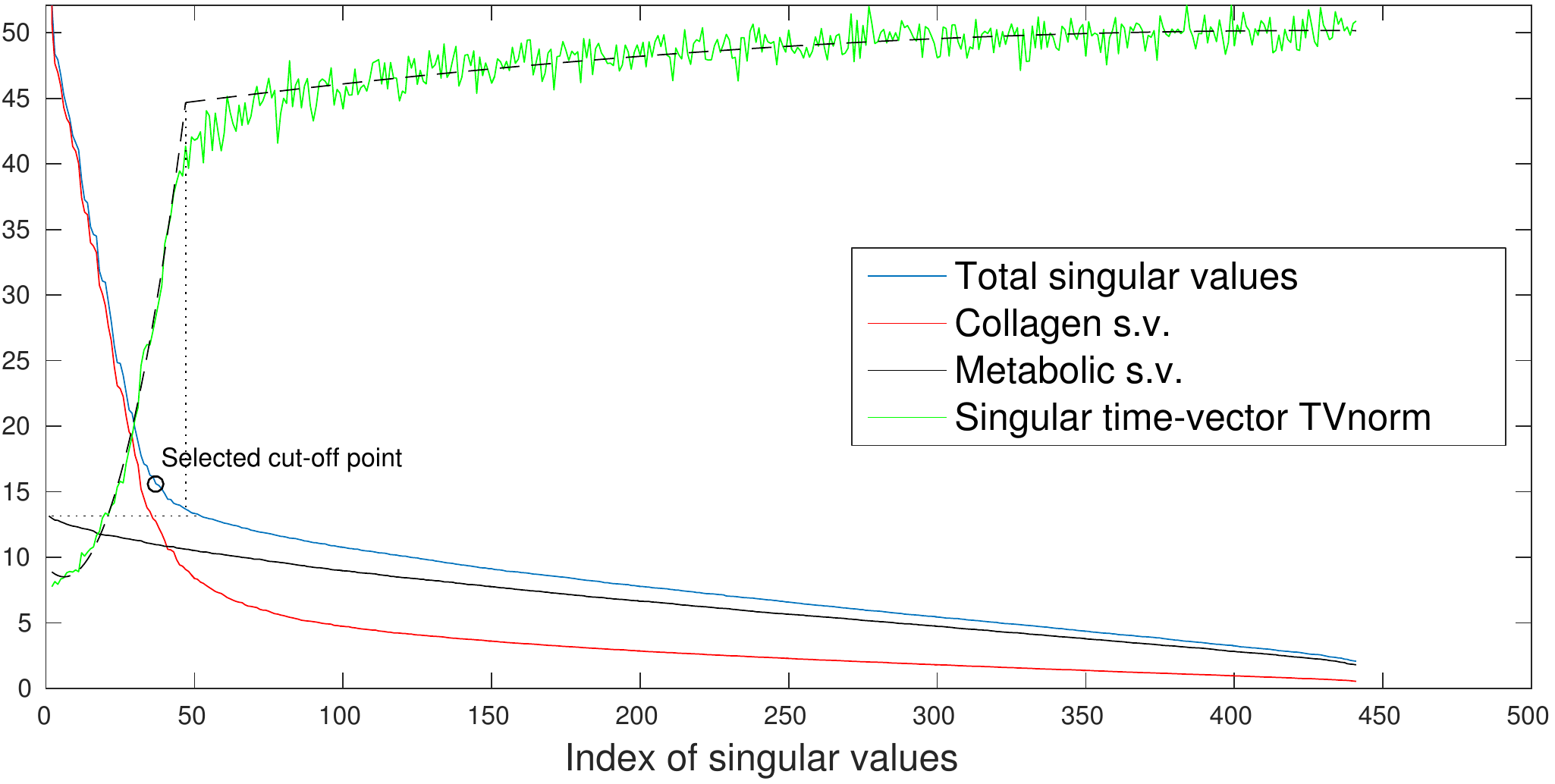}
	\caption{ Same plot as in Figure \ref{fig:singularValueDecay}, but including the total variation of the singular time-vectors of the total signal. The total variation is scaled to fit the plot with the singular values.}
	\label{fig:cutoffTechnique}
\end{figure}

\subsection{ Signal reconstruction }

Employing the cut-off criteria in subsection \ref{subsec:cut-off} and formula \eqref{eq:reconstructFormula} to our simulation, we can reconstruct the metabolic activity. In Figure \ref{fig:reconstructedMetabolic} we have on the left-hand side the best possible reconstruction using the SVD technique, that is the one we could do if we could isolate completely the signal $A_m$ from the total signal $A$. On the right-hand side, we have the actual reconstruction. It is worth mentioning that we are not able to reconstruct the exact metabolic map, as formula (\ref{eq:reconstructFormula}) is used on the simulated media, and thus the image obtained out of the isolated signal $A_m$ is the one we are aiming to reconstruct.

\begin{figure}[h!]
	\centering
	\begin{minipage}[c]{.45\textwidth}
		\centering
		\tiny Best possible reconstruction
	\end{minipage}
	\hspace{.5cm}
	\begin{minipage}[c]{.45\textwidth}
		\centering
		\tiny Achieved reconstruction
	\end{minipage}
	\\
	\begin{minipage}[c]{.45\textwidth}
		\centering
		\includegraphics[width=\linewidth]{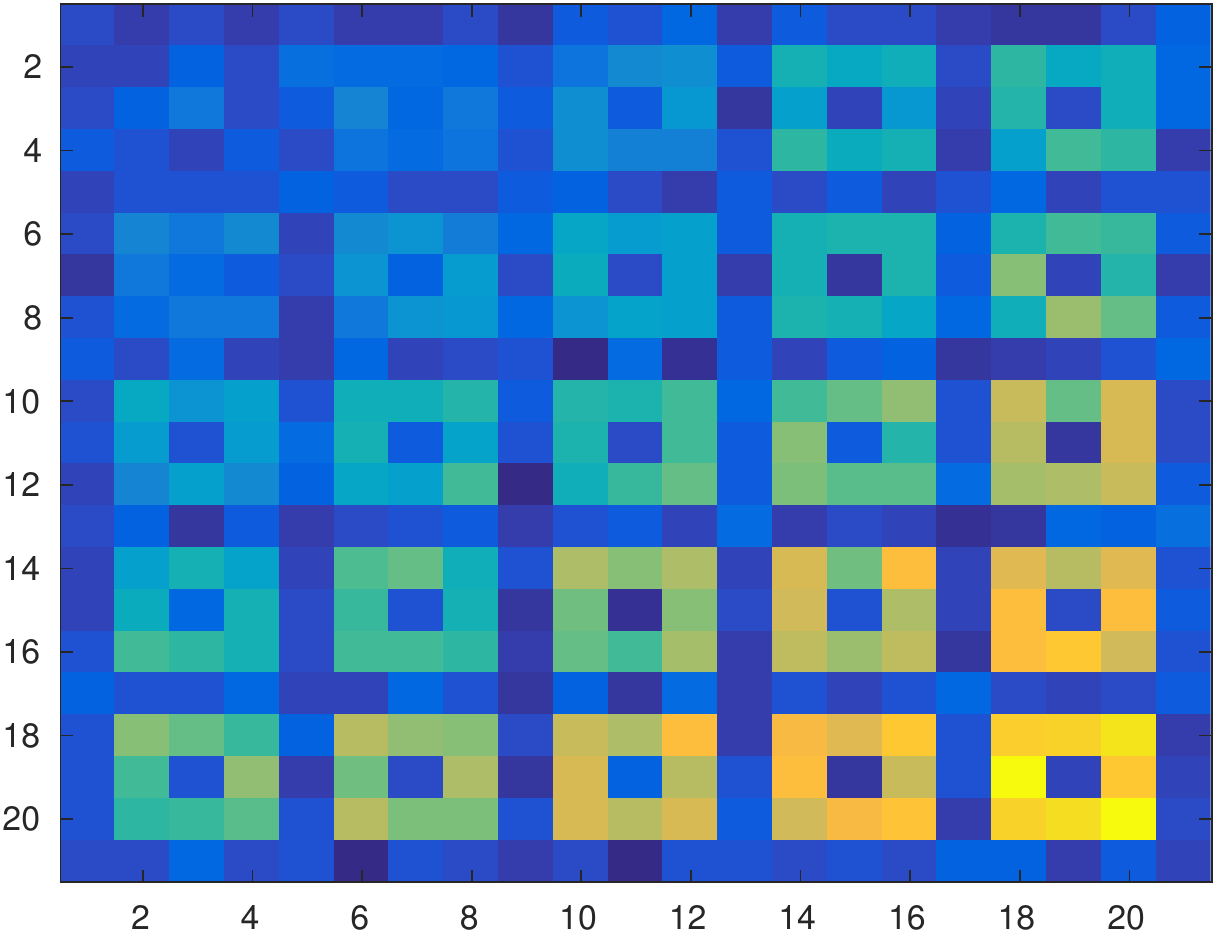}
	\end{minipage}
	\hspace{.5cm}
	\begin{minipage}[c]{.45\textwidth}
		\centering
		\includegraphics[width=\linewidth]{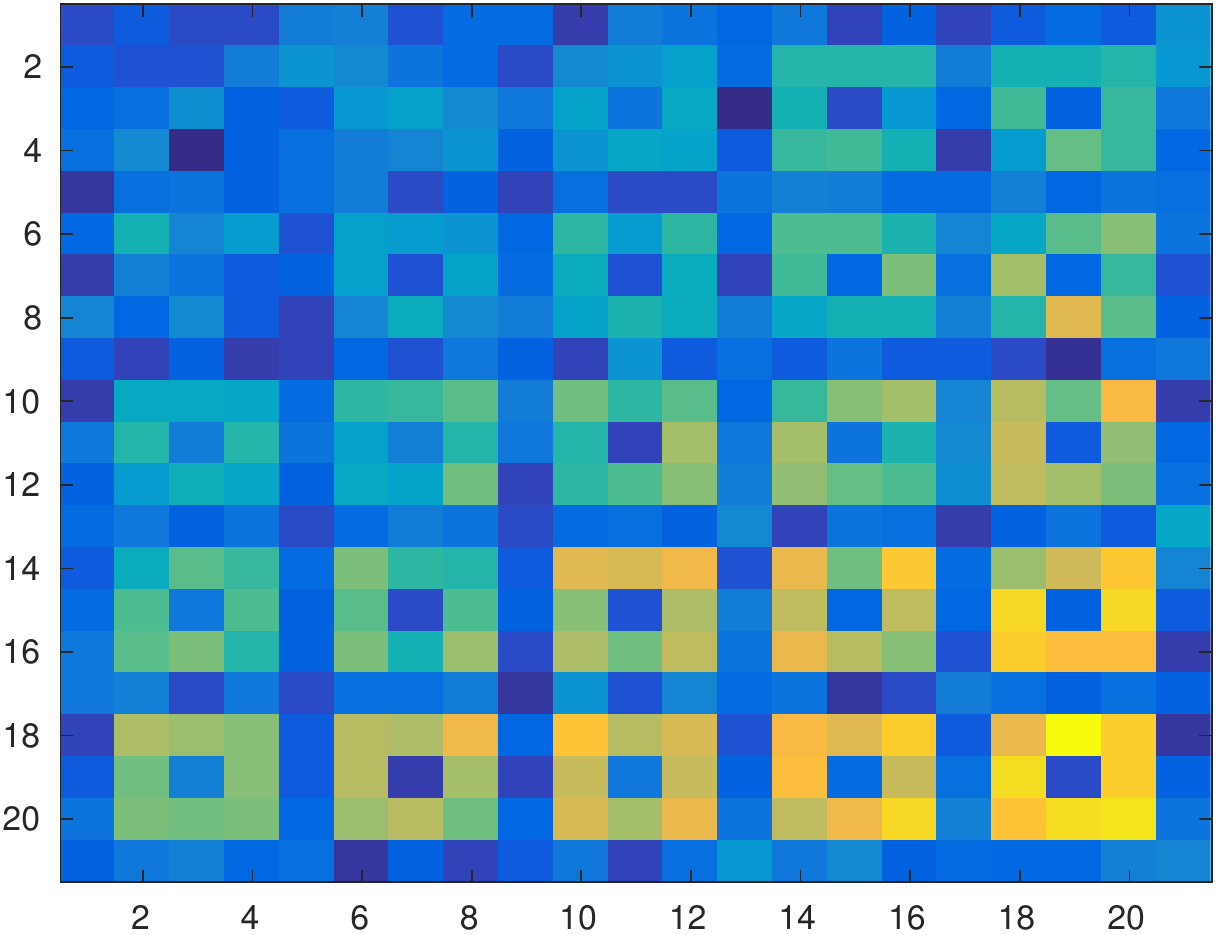}
	\end{minipage}
	\caption{ Reconstruction of the metabolic map presented in Figure \ref{fig:metabolicMap}. The left image correspond to using directly formula (\ref{eq:reconstructFormula}) on the isolated $A_m$ signal. The right-hand side image correspond to using our reconstruction method on the total signal. Once the images are normalized, the committed error with respect to the original metabolic map is $0.011$ and $0.017$, respectively}
	\label{fig:reconstructedMetabolic}
\end{figure}

\subsection{ Discussion and observations}

Since the SVD uses information of all pixels simultaneously to filter out the collagen signal, this technique works better the larger the considered image size is, as the main point is to use the joint information of all the pixels in the image, in contrast to frequency filtering that considers only pointwise information. Numerically, this effect is notorious, as the larger the image size, the more clustered are the singular space-vectors associated to the metabolic activity and thus it is easier to filter out the collagen signal.

With respect to the time samples, it is observed that the filtering process degrades if too many time samples are considered. When this happens (for our 21x21 grid size, this is above 1000 time samples), the singular values of the collagen signal start decaying in a slower rate, accomplishing a less clustered behavior of the metabolic singular space-vectors, and thus achieving a worse signal separation. Hence if there are high available amounts of time samples, one possible recommendation is to do several reconstructions using subsets of these time samples and then averaging the results.

\section{Conclusion}

In this paper, we performed a mathematical analysis of extracting useful information for sub-cellular imaging based on dynamic optical coherence tomography. By using a novel multi-particle dynamical model, we analyzed the spectrum of the operator with the intensity as an integral kernel, and shown that the dominant collagen signal is rank-one.  Therefore, a SVD approach can theoretically separate the metabolic activity signal from the collagen signal. We proved that the SVD eigenvectors are good approximation to the collagen signal, proving that the SVD approach is feasible and reliable as a method to remove the influence of collagen signals. And we also discovered a new formula that gives the intensity of metabolic activity from the SVD analysis. This is further confirmed by our numerical results on simulated data sets.

\appendix
\section{Non-orthogonality of the eigenvectors of $F_{cc}$, $F_{cm}$, and $F_{mc}$} \label{App:AppendixA}

In this appendix we will illustrate the fact that the eigenvectors of the kernels $F_{cc}(x, y)$, $F_{cm}(x, y)$ and $F_{mc}(x, y)$ are in general not orthogonal. Since all of them have variable separable forms with respect to $x$ and $y$, which is the basis of our analysis, so here we only prove the nonorthogonality between eigenvectors of the kernels $F_{cc}(x, y)$ in (\ref{kernelFdmm}) and $F_{cm}(x, y)$ in (\ref{kernelFdmc}).

Let $A$ be the matrix obtained from discretizing the signal $\Gamma_{ODT}$. The singular values of $A$ are the square roots of the eigenvalues of the matrix $A^*A$, and the singular vectors of $A$ are the corresponding eigenvectors of $A^*A$. We notice that $A^*A$ is a discretization of the integral kernel $F(x,y)$. We first demonstrate the relation between kernels with variable separable forms and eigenvectors.

\begin{lemma}\label{eigenvector-varisepe}
For any function $f(x, y)$ where $x$ and $y$ belong to $R^d$ with $d$ being the space dimension, if there exist functions $f_1(x)$ and $f_2(y)$, such that $f(x, y)=f_1(x)f_2(y)$, then $f_1(x)$ and  $f_2(y)$ are the eigenvectors of the integral operator $T$ with kernel $f(x, y)$.
\end{lemma}
\begin{proof}
Define the operator $T$ with the kernel $f(x, y)$ to be $(Th)(x)=\int f(x, y)h(y)dy$. Using the variable separation $f(x, y)=f_1(x)f_2(y)$, we obtain
\begin{equation*}
(Th)(x)=\int f_1(x)f_2(y)h(y)dy=f_1(x)\int f_2(y)h(y)dy.
\end{equation*}
Therefore, it is clear that the operator $T$ has eigenvector $f_1(x)$, where $\int f_2(y)f_1(y)dy$ is the associated eigenvalue. Similarly, $f_2(y)$ is also the eigenvector of $T$, where  $\int f_1(x)f_2(x)dx$ is the associated eigenvalue.
\end{proof}

Denote the functions $\varphi^{c}(x)$ and $\varphi^{m}(x)$ by
\begin{align*}
\begin{split}
\varphi^{c}(x)&=K_{c_1}(x)q_c(x),\\
\varphi^{m}(x)&=\int_{[-L,L]^2\times[0, T]}\mathcal{F}(S_{0}K_{c_2})(\frac{4\pi \bar{n}z_2}{c})\\
&\times \mathcal{F}(S_{0}K_{m})(x,-\frac{4\pi \bar{n}z_1}{c})p_m(x, z_1, t)dz_1dz_2 dt
\end{split}
\end{align*}

Then the kernels $F_{cc}$ and $F_{cm}$ can be written as
\begin{align*}
\begin{split}
F_{cc}(x,y)&=C_1\varphi^{c}(x)\varphi^{c}(y),\\
F_{cm}(x,y)&=C_2\varphi^{c}(x)\varphi^{m}(y),
\end{split}
\end{align*}
where $C_1$ and $C_2$ are constants.

Applying Lemma \ref{eigenvector-varisepe} to the kernels $F_{cc}$ and $F_{cm}$, we know that the corresponding eigenvectors are $\varphi^{c}$ and $\varphi^{m}$ respectively.

Since this integral $\int\varphi^{c}(x)\varphi^{m}(x)dx$ depends much on the random term $p_m(x,z,t)$, it will not be zero almost all of the time. Hence, in our construction, the vectors $\varphi^{c}$ and $\varphi^{m}$ are in general not orthogonal.

\bibliographystyle{siam}
\bibliography{mc}

\end{document}